\documentclass[12pt,oneside]{amsart}
\usepackage[utf8]{inputenc}
\usepackage[english]{babel}

\usepackage[margin=1in]{geometry}

\usepackage{multicol}

\usepackage{amsmath}
\usepackage{amsthm}
\usepackage{amssymb}
\usepackage{graphicx}
\usepackage{mathrsfs}
\usepackage[colorinlistoftodos]{todonotes}
\usepackage[colorlinks=true, allcolors=black]{hyperref}
\usepackage{comment}
\usetikzlibrary{graphs}
\usepackage{float}
\usepackage{xcolor}
\usepackage{soul}
\usepackage{tikz-cd}

\usepackage{phaistos}

\newcommand{\showcomments}{yes}
\newcommand{\showrefcomments}{yes}
\renewcommand{\showrefcomments}{no}

\newsavebox{\commentbox}
{\ifthenelse{\equal{\showcomments}{yes}}%
{\footnotemark
        \begin{lrbox}{\commentbox}
        \begin{minipage}[t]{1.5in}\raggedright\sffamily\tiny
        \footnotemark[\arabic{footnote}]}
{\begin{lrbox}{\commentbox}}}%
{\ifthenelse{\equal{\showcomments}{yes}}%
{\end{minipage}\end{lrbox}\marginpar{\usebox{\commentbox}}}
{\end{lrbox}}}

{\ifthenelse{\equal{\showrefcomments}{yes}}%
{\footnotemark
        \begin{lrbox}{\commentbox}
        \begin{minipage}[t]{1.25in}\raggedright\sffamily\tiny
        \footnotemark[\arabic{footnote}]}
{\begin{lrbox}{\commentbox}}}%
{\ifthenelse{\equal{\showrefcomments}{yes}}%
{\end{minipage}\end{lrbox}\marginpar{\usebox{\commentbox}}}
{\end{lrbox}}}

\newtheorem{thm}{Theorem}[section]

\newtheorem{theorem}[thm]{Theorem}
\newtheorem{corollary}[thm]{Corollary}
\newtheorem{lemma}[thm]{Lemma}
\newtheorem{proposition}[thm]{Proposition}

\newtheorem{claim}[thm]{Claim}

\newtheorem{innercustomthm}{Theorem}
\newenvironment{customthm}[1]
  {\renewcommand\theinnercustomthm{#1}\innercustomthm}
  {\endinnercustomthm}
  
  \newtheorem{innercustomcor}{Corollary}
\newenvironment{customcor}[1]
  {\renewcommand\theinnercustomcor{#1}\innercustomcor}
  {\endinnercustomcor}

\newtheorem{innercustomquest}{Question}
\newenvironment{customquest}[1]
  {\renewcommand\theinnercustomquest{#1}\innercustomquest}
  {\endinnercustomquest}

\newtheorem{conjecture}[thm]{Conjecture}
\newtheorem{question}[thm]{Question}

\theoremstyle{definition}

\newtheorem{definition}[thm]{Definition}

\newtheorem{prob}[thm]{Problem}

\theoremstyle{remark}

\newtheorem{remark}[thm]{Remark}

\newcommand{\scname}[1]{\text{\sf #1}}
\newcommand{\area}{\scname{Area}}

\usepackage{xcolor}

\title{Taut smoothings of arcs and curves}

\author{Macarena Arenas}
\address{DPMMS, Centre for Mathematical Sciences, Wilberforce Road, Cambridge, CB3 0WB, UK
 and Clare College, University of Cambridge, Cambridge, CB2 1TL, UK}
\email{mcr59@dpmms.cam.ac.uk}

\author{Max Neumann-Coto}
\address{Instituto de Matemáticas, Universidad Nacional Autónoma de México, Ciudad Universitaria, 04510, Mexico City, Mexico}
\email{max.neumann@im.unam.mx}

\thanks{The first author was supported by the Denman Baynes Junior Research Fellowship at Clare College, Cambridge.}

\begin{document}

\begin{abstract}
    We study the geometric and combinatorial effect of smoothing an intersection point in a collection of arcs or curves on a surface. We prove that  all taut arcs with fixed endpoints and all taut 1-manifolds with at least two non-disjoint components on an orientable surface with negative Euler characteristic admit a taut  smoothing, and also that all taut arcs with free endpoints admit a smoothing that is either taut or becomes taut after removing at most one  intersection. We deduce that for every Riemannian metric on a surface, the shortest properly immersed arcs  with at least $k$ self-intersections have exactly $k$ self-intersections when the endpoints of the arc are fixed, and at most $k+1$ self-intersections otherwise, and that the arc length spectrum is ``coarsely ordered" by self-intersection number. Along the way, we obtain partial analogous results in the case of curves.
\end{abstract}

\dedicatory{In memory of Peter Scott}
\maketitle

\section{Introduction}

 This paper is about homotopy classes of properly immersed 1-manifolds (collections of immersed curves and arcs) on an orientable surface $\Sigma$ with negative Euler characteristic, and about the connection between distinct homotopy classes resulting from certain surgery operations, known as smoothings.
 A smoothing is a way of resolving an intersection point by deleting a neighbourhood of the point and joining the ends of the remaining arcs in pairs, as shown in Figure~\ref{example intro}. 
 These purely topological operations have deep geometric consequences, and provide a way of relating
 the lengths and intersection numbers of geodesics on the surface.

 \begin{figure}[h]
\centering
\includegraphics[height=0.4in]{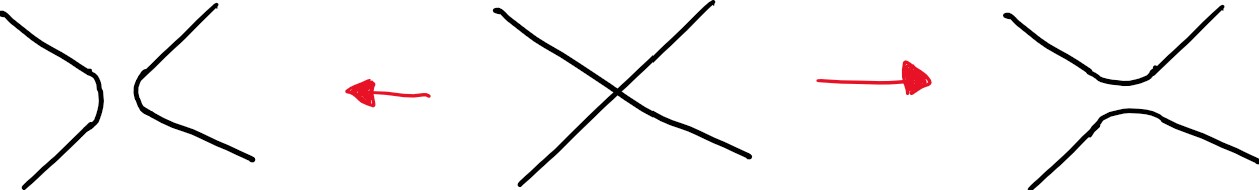}
\caption{Smoothings.}
\label{example intro}
\end{figure}

We focus on \emph{taut} 1-manifolds, that is, immersed 1-manifolds on $\Sigma$ that 
have the minimal number of self-intersections within their free homotopy class (rel boundary). In this work  1-manifolds are not assumed to be connected, and may be comprised of both closed curves and arcs. 

A taut 1-manifold $\alpha$ \emph{admits a taut smoothing} if there is a double point $p$ and a smoothing at $p$ resulting in a taut 1-manifold. So if  $\alpha$  has $k$ self-intersections, the result of the smoothing cannot be homotoped to have less than $k-1$ self-intersections

 We emphasise that in this paper, smoothings are performed at one intersection point at a time. It is a different problem altogether to try to understand the connection between a given $1$-manifold and the possible results of smoothing multiple intersection points at once.

\begin{customquest}{1}\label{qn:1 intro}
Does every taut 1-manifold on a surface admit a taut smoothing? 
\end{customquest}

The existence of a taut smoothing for a taut 1-manifold implies the existence of a taut smoothing for all taut 1-manifolds in the same free homotopy class
(where, in the case of arcs,  the endpoints may or may not be fixed).

The result of a smoothing at a double point $p$ does not always have the same number of connected components as the original 1-manifold. For instance, in the case of a curve or an arc, at every double point, precisely one of the smoothings produces again a curve or an arc. This suggests the following strengthening of Question~\ref{qn:1 intro}, which is particularly important for applications.

\begin{customquest}{1'}\label{qn:1' intro}
Does every taut 1-manifold admit a taut smoothing with no more components than the original 1-manifold? 
In particular, can every taut curve be smoothed to a taut curve, and can every taut arc be smoothed to a taut arc?
\end{customquest}

The main theorems of this work settle Question~\ref{qn:1' intro} in the case of arcs:

\begin{customthm}{A}
[Theorem~\ref{thm: arcs}+Theorem~\ref{cor:arcos libres}]\label{thm: arcs intro}
If $\alpha$ is a non-simple arc on an orientable surface $\Sigma$ that is taut fixing its endpoints, then some smoothing of $\alpha$ produces an arc that is taut fixing its endpoints.
If $\alpha$ is taut  with free endpoints and has $k$ self-intersections, then $\alpha$  has a smoothing that produces an arc that cannot be homotoped to have less than $k-2$ self-intersections. 
\end{customthm}

For arcs that are taut with free endpoints, it is not always the case that one can find a taut smoothing. 
One example is illustrated in Figure~\ref{fig: freearcs}.

There are examples of taut curves and arcs with an arbitrarily large number of self-intersections for which only one smoothing produces a taut curve or arc. 
There are also examples for which some smoothings produce curves or arcs that  are homotopically simple.
Figure~\ref{fig: onlyoneforarcs} hints at how to construct such behaviour.

 \begin{figure}[h]
\centering
\includegraphics[height=.9in]{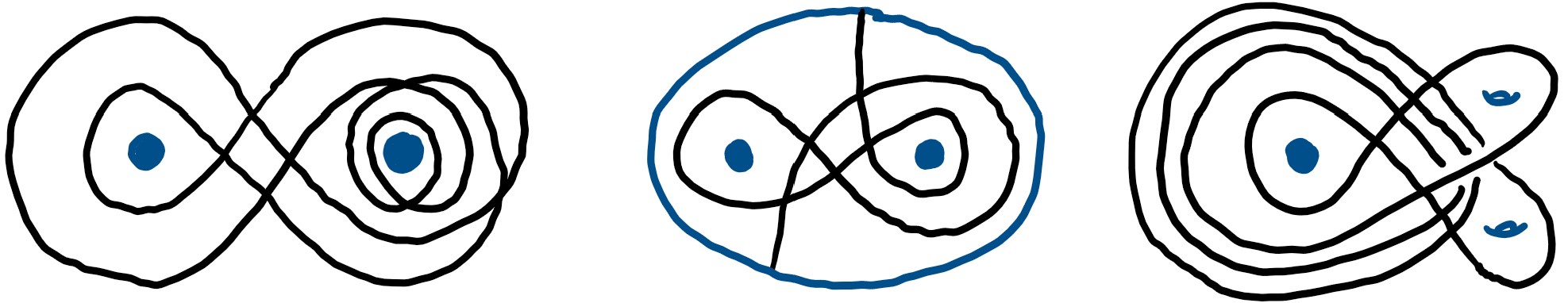}
\caption{Left and centre: A taut curve and a taut arc both of which  only have one connected taut smoothing.  Right: a curve with a smoothing that removes all self-intersections. These examples can be generalised to have an arbitrarily large number of self-intersections.}
\label{fig: onlyoneforarcs}
\end{figure}

A positive answer to Question~\ref{qn:1' intro}  for curves or arcs  implies, respectively, a positive answer (see the proof of Corollary~\ref{cor: arcs geometryfull}) to the following question for curves or arcs:

\begin{customquest}{2}\label{qn:2 intro}
Is it true that on every Riemannian surface the shortest geodesics with at least $k$ self-intersections have exactly $k$ self-intersections? And that the shortest geodesic arcs with at least $k$ self-intersections have exactly $k$ self-intersections?
\end{customquest}

A shortest geodesic with at least $k$ self-intersections  is also  called a \emph{k-systole}.

For curves,  Question~\ref{qn:2 intro} is quite well-known, and has been asked many times in the literature (see for instance~\cite{Basmajian93,Bus92}). In this case, part of the significance of Question~\ref{qn:2 intro}   stems from the fact that it translates to a monotonicity question for the ordered \emph{length spectrum} of a surface -- thas is, the ordered set
of  lengths of shortest geodesics in each  free homotopy class of curves. It is clear that an (ordered) \emph{arc length spectrum} can also be defined for geodesic arcs (both with fixed and with free endpoints), and that Question~\ref{qn:2 intro} for arcs also relates to the monotonicity of this spectrum.

For each Riemannian metric on $\Sigma$, it is easy to show that the shortest essential geodesics are simple if $\Sigma$ is not a 3-holed sphere, and in~\cite{Bus92}  Buser proved that the shortest non-simple geodesics always have $1$ self-intersection. It can also be shown via a simple counting argument 
that for each Riemannian metric on $\Sigma$ there are infinitely many $k$'s (dependent on the metric!) such that the shortest geodesics with at least $k$ self-intersections have exactly $k$ self-intersections. For compact surfaces, no other results characterising the intersection numbers of $k$-systoles independently of the metric are present in the literature.  See~\cite{EP20} for a related discussion.

Theorem~\ref{thm: arcs intro} shows:

\begin{customcor}{B}[Corollary~\ref{cor: arcs geometryfull}]\label{cor: arcs geometry intro} For every orientable Riemannian surface $\Sigma$ with convex boundary, the shortest taut arcs with  distinct fixed endpoints and at least $k$ self-intersections have exactly $k$ self-intersections. The shortest taut arcs with free endpoints and at least $k$ self-intersections have at most $k+1$ self-intersections.
\end{customcor}

This settles Question~\ref{qn:2 intro}  in the case of arcs, and constitutes, to the best of our knowledge, the only  result to date computing the minimal self-intersection numbers of shortest  geodesic arcs or curves precisely for all values of $k$ and for all compact orientable surfaces: previous results in this direction either  describe asymptotic behaviour,  give  bounds depending on $k$ for the number of self-intersections of a $k$-systole, or produce exact results only in the case of cusped hyperbolic surfaces  -- where the geometry is rather special --  and for large enough values of $k$ dependent on the metric. See Subsection~\ref{subsec: related} for more details on related results.

Corollary~\ref{cor: arcs geometry intro} puts no restriction on the Riemannian metric on $\Sigma$: the only properties  being used are that primitive geodesics intersect transversally, and that rounding-out corners reduces length.

\begin{figure}[h]
\centering
\includegraphics[height=.63in]{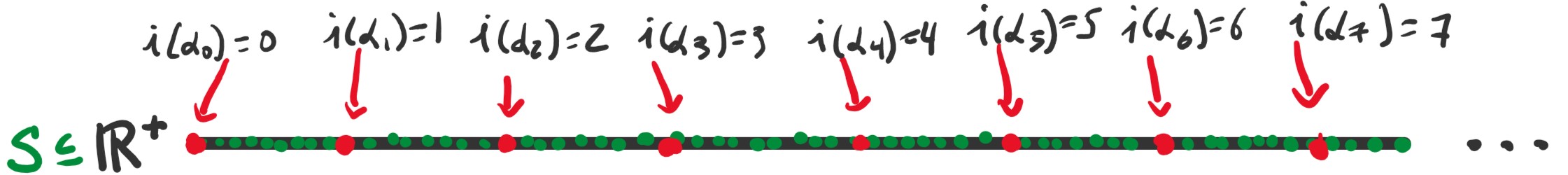}
\caption{For all Riemannian metrics, Corollary~\ref{cor: arcs geometry intro} implies that
there is an increasing sequence of points in the arc length spectrum $S$ that is graded by intersection numbers.}\label{fig: spectrum}
\end{figure}

We note that, both for curves and for arcs, a positive answer to Question~\ref{qn:1' intro} is strictly stronger than a positive answer to Question~\ref{qn:2 intro}. 
In fact, the first implies that for each free homotopy class (resp. free homotopy class rel endpoints) of curves (resp. arcs) with $k$ self-intersections, there is another class of curves (resp. arcs) with $k-1$ self-intersections such that for every Riemannian metric on $\Sigma$, the shortest geodesic (resp. geodesic arc) in the first class is longer that the shortest geodesic (resp. geodesic arc) in the second class. See Corollaries~\ref{cor:universalarc}  
 and~\ref{cor:universalarcfree} for the precise statements.

By iteratively applying Theorem~\ref{thm: arcs intro} and Theorem~\ref{cor:arcos libres}, our results imply that the arc length spectrum is ``coarsely ordered" by intersection numbers.

Along the way to proving Theorem~\ref{thm: arcs intro}, we answer Question~\ref{qn:1' intro} in the positive for taut curves satisfying a fairly general condition:

\begin{customthm}{C}
[Lemma~\ref{lem: smoothing invariants htpy class}+Corollary~\ref{cor:maincurvessmooth} +Theorem~\ref{thm:mainsmooth}]\label{thm:taut smooth curve if no trisq intro}
     Let $\Sigma$ be an orientable surface.  If a taut curve $\alpha$ in $\Sigma$ satisfies that either 
     \begin{enumerate}
         \item $\alpha$ has an unoriented corner that is not the corner of a  quadrilateral or a pentagon, or
         \item   $\alpha$ admits a decomposition as $\alpha=\beta \gamma$ where $\gamma$ is a  subloop that only intersects $\beta$ at its endpoints, and no triangle of $\alpha$ has $\gamma \cap \beta$ as a vertex,
     \end{enumerate}
     then  every taut curve homotopic to $\alpha$ can be smoothed to a taut curve (after potentially deleting a simple curve component).  
\end{customthm}

While the conditions in Theorem~\ref{thm:taut smooth curve if no trisq intro}  seem somewhat restrictive (see Figure~\ref{fig: regions}), many curves satisfy them or can be homotoped to satisfy them. For instance, if $\alpha$ does not fill the surface, then some complementary region $R$ is not a disc, so if the boundary of $R$ is not cyclically oriented, then Theorem~\ref{thm:mainsmooth} can be applied to $\alpha$. 

 \begin{figure}[h]
\centering
\includegraphics[height=1.1in]{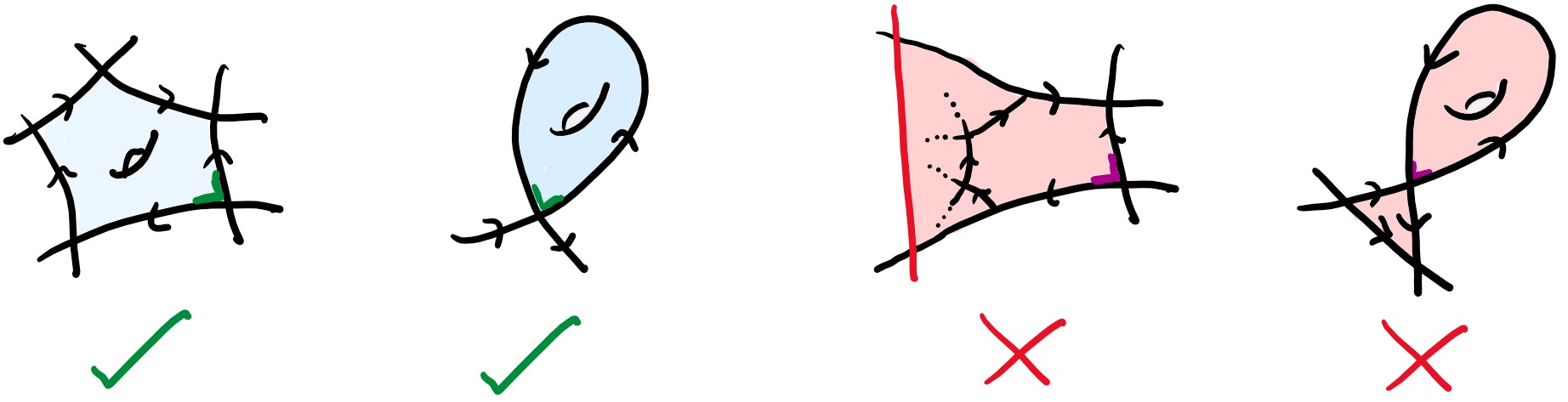}
\caption{The corners marked on the two leftmost figures satisfy the hypotheses of Theorem~\ref{thm:taut smooth curve if no trisq intro}; the corners marked on the two rightmost figures do not.}
\label{fig: regions}
\end{figure}

The proofs of  Theorems~\ref{thm: arcs intro} and~\ref{thm:taut smooth curve if no trisq intro} use Proposition~\ref{prop:pathologies}, which  classifies small bigons and monogons arising from smoothings, together with Theorem~\ref{thm:mainsmooth}, which is the main technical theorem of this work, and  a notion of ``combinatorial angle" which has the advantage of being independent of the metric and of the configuration of the curve. One subtlety that complicates arguments is that when  homotoping a 1-manifold to find a ``better" configuration, we need to be able to guarantee the existence of a sequence of triangular moves that preserves a number of additional  properties (see for instance Lemma~\ref{lem: moveable triangles}). 

The proofs of Theorems~\ref{thm: arcs intro} and~\ref{thm:taut smooth curve if no trisq intro}  have the advantage of being constructive, so they not only show the existence of taut smoothings under the appropriate hypotheses, but also show how to find such smoothings explicitly. 
Finally, we point out that our results do not require the surface to be of finite type.

Using again Proposition~\ref{prop:pathologies} and some hyperbolic geometry, we additionally answer Question~\ref{qn:1 intro} completely in the case of 1-manifolds with at least $2$ non-disjoint connected components.

\begin{customthm}{D}[Theorem~\ref{prop:taut pair}]\label{thm:multi-intro} 
If a taut 1-manifold $\alpha$ on an orientable surface $\Sigma$ contains $2$ non-disjoint primitive components, then $\alpha$ has a smoothing that is taut fixing its endpoints.
\end{customthm}

Geometric consequences analogous to Corollary~\ref{cor: arcs geometry intro}  can also be derived from Theorems~\ref{thm:taut smooth curve if no trisq intro} and~\ref{thm:multi-intro} for certain classes of closed geodesics, for multi-arcs, and for geodesic 1-manifolds with intersecting components.

\subsection{Related works}\label{subsec: related}

  Question~\ref{qn:2 intro} has been investigated by a number of authors. In~\cite{EP20}, Erlandsson--Parlier show that the number of self-intersections of $k$-systoles can be bounded above by a linear function of $k$, and also show that in the case of orientable finite-type surfaces with cusps, the number of self-intersections of $k$-systoles behaves asymptotically like $k$ as $k$ grows. In~\cite{Vo22}, Vo shows that  if the surface $\Sigma$ is hyperbolic and has at least one cusp, then for  sufficiently large $k$ depending on the hyperbolic metric on $\Sigma$, the shortest geodesics with at least $k$ self-intersections have exactly $k$ self-intersections. 
 In~\cite{BPV22}, Basmajian--Parlier--Vo use this result to bound, in terms of $k$, the length of  $k$-systoles for sufficiently large $k$.
 For surfaces of small complexity, there are also some computational results~\cite{ChasPhillips10,ChasPhillips12, ChasMcMullenPhillips19, DGS21} investigating the relation between (geometric or combinatorial) length and self-intersection numbers for small enough values of $k$. 

Taut smoothings, and more generally smoothing without monogons, are also related to the study of skein algebras~\cite{Przytycki91,Turaev89}, and in particular to the question of finding positive bases~\cite{ThurstonD14} for them. See Corollary~\ref{cor:nomomogons} and the ensuing paragraph.
 
 In the preprint~\cite{DTh}, the same author proves the \emph{Smoothing Lemma}, a formula that relates the number of intersections between a simple curve $s$ and an immersed curve $c$ with the number of intersection between $s$ and the smoothings of $c$. He then uses this result to study Dehn--Thurston coordinates.

Recent work of Martinez-Granado--Thurston~\cite{MGT21} explores continuous functions of geodesic currents, and shows that in many cases, these behave similarly to the functions of closed curves that they extend. To determine whether the functions on geodesic currents are ``nice'' they investigate their behaviour under performing smoothings. 

In a rather different direction, it was pointed out to us by Ivan Smith that there is work in the symplectic setting  dealing with similar questions, in particular in relation with combinatorial Floer homology and tautologically unobstructed $1$-manifolds~\cite{AurouxSmith21,dSRS14,HKK17}. In that framework, however, it is important that isotopies be Hamiltonian, or at the very least area preserving;  we have not attempted to engage with these properties in the current work.

\subsection{Structure of the paper}
In Section~\ref{sec: background} we introduce the necessary terminology and make a few basic observations. In Section~\ref{sec: classifying} we classify the possible bigons and monogons that can be produced by smoothings, and  explore the relationship between taut smoothings of multi-component 1-manifolds and maximal hyperbolic angles. The main result in this section is Theorem~\ref{thm:multi-intro}. In Section~\ref{sec:taut smoothing} we prove the main results of this work, namely Theorems~\ref{thm: arcs intro} and~\ref{thm:taut smooth curve if no trisq intro}. In Section~\ref{sec: geometry} we discuss some geometric consequences of the results in the previous sections.  Finally, in Section~\ref{sec: further}, we state and  discuss some related questions.

\subsection*{Acknowledgements} The second author is grateful to Joel Hass for introducing him to some of these questions. The first author thanks Ivan Smith for pointing out the relation with symplectic topology, and Henry Wilton  for helpful conversations. The authors thank Dylan Thurston for pointing out the connection between Question~\ref{qn:1' intro} and the main theorem in~\cite{ThurstonD14}.

\section{Preliminaries}\label{sec: background}

\subsection{Minimal intersections and configurations}\label{subsec: basicstuff}

In this paper, $\Sigma$  always denotes an orientable surface with negative Euler characteristic. We refer to the closed properly immersed curves on $\Sigma$ simply as \emph{curves}, and to the properly immersed arcs as \emph{arcs}. A \emph{1-manifold} in $\Sigma$ is a finite collection of curves and arcs. All  \mbox{1-manifolds} are assumed to be in \emph{general position}, so they intersect transversally at double points, all of which lie in the interior of $\Sigma$. A 1-manifold is \emph{simple} if it has no self-intersections. 

A \emph{taut} curve in $\Sigma$ is an essential curve with the minimum number of self-intersections within its free homotopy class. Similarly, a \emph{taut} arc is an essential arc with the minimum number of self-intersections, but we must specify whether the endpoints are fixed or allowed to move freely along $\partial \Sigma$.
Finally, a 1-manifold $\alpha$ is \emph{taut} if each component self-intersects minimally, and distinct components intersect minimally (either with fixed or free endpoints).

We will assume that all curves are primitive (not homotopic to proper powers of other curves) and, more generally, that all 1-manifolds are primitive (so every curve component is primitive).

By work of Freedman--Hass--Scott~\cite{FHS82} and Neumann-Coto~\cite{NC01}, a primitive 1-manifold on a surface $\Sigma$ is taut if and only if there exists a Riemannian metric\footnote{If one restricts to hyperbolic metrics, it is true that every hyperbolic 1-manifold is taut, but not that every taut configuration can be realised by geodesics in such a metric~\cite{HassScott99}.} on $\Sigma$ for which its components are the shortest geodesic curves and arcs in their free homotopy classes.

\subsection{Smoothings}

\begin{definition} 
Let $\alpha$ be a 1-manifold. A \emph{smoothing of $\alpha$ at a double point $p$} is the result of deleting a neighbourhood $N_\epsilon(p)$ of $p$ in $\Sigma$ and connecting pairs of endpoints of $\alpha - N_\epsilon(p)$, as in Figure~\ref{fig: smoothingsdetail}. 
\end{definition}

 \begin{figure}[h]
\centering
\includegraphics[height=0.5in]{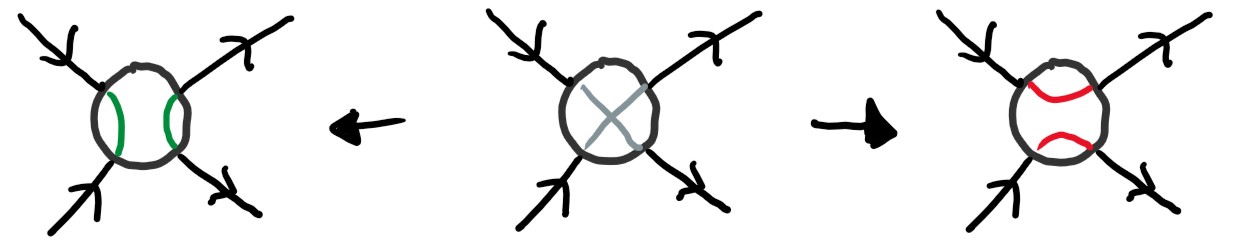}
\caption{The two possible kinds of smoothings.}
\label{fig: smoothingsdetail}
\end{figure}

We say that a smoothing is \emph{taut} if it produces a taut 1-manifold.

At times it will be convenient to abuse notation and use the term \emph{smoothing} both for the operation described above and for its result.

 For curves, smoothings can be described more concretely as follows. Let $\alpha:S^1 \rightarrow \Sigma$ be a curve in $\Sigma$, let $p$ be a double point of $\alpha$,  let $S^1=A\cup B$ where $A$ and $B$ are the intervals whose endpoints map to $p$ via $\alpha$, and finally let $\bar \beta$ denote the arc or curve $\beta$ traversed in the opposite direction. 

Then a  \emph{smoothing of $\alpha$ at $p$} is either
\begin{enumerate}
    \item \label{smt 1} the concatenation
 $\alpha ^+= \alpha |_A \ \bar\alpha |_B$, or
 \item \label{smt 2}
 the union $\alpha ^-= \alpha |_A \cup \alpha |_B$.
\end{enumerate}

Whether a smoothing at a double point $p$ is of type~\eqref{smt 1} or~\eqref{smt 2} depends on how the endpoints of $\alpha - N_\epsilon(p)$ are being rejoined:  $\alpha$ comes with an orientation, which can be read locally on the arcs around $p$. 
If the smoothing connects pairs of arcs with consistent orientation, as in the right of Figure~\ref{fig: smoothingsdetail}, then the result is disconnected, and thus of type~\eqref{smt 2}; if the smoothing connects pairs of arcs with opposite orientations, as in the left of Figure~\ref{fig: smoothingsdetail}, then the result is connected, and thus of type~\eqref{smt 1}. 
This is because of how the $4$ endpoints of the arcs around $p$ in $N_\epsilon(p)$ connect to the rest of the curve. We discuss this in more detail in Subsection~\ref{subsec: repel}.

The following result follows from the work of Hass and Scott ~\cite{HassScott85}:

\begin{lemma}\label{lem: smoothing invariants htpy class}
    If $\alpha$ and $\alpha'$ are two freely homotopic taut 1-manifolds with primitive and non-homotopic components on a surface $\Sigma$, then
    \begin{enumerate}
        \item there is a natural correspondence between the intersection points of $\alpha$ and $\alpha'$,
        \item each smoothing of $\alpha$ is freely homotopic to the corresponding smoothing of $\alpha'$.
    \end{enumerate}
\end{lemma}

\begin{proof}
It was shown in~\cite{HassScott85} that under the assumptions there is a homotopy from $\alpha$ to $\alpha'$ that does not add or remove intersections at any time. So each intersection point can be traced along the homotopy from $\alpha$ to $\alpha'$.
The homotopy consists of a sequence of triangle moves at innermost triangles. To show that the smoothings at an intersection point of $\alpha$ and the corresponding point of $\alpha'$ are freely homotopic we need to see that the smoothing operation commutes with the triangle moves. This is clear if the smoothing is not performed at a vertex of the triangle $\Delta$.
If it is, the smoothing and the triangle move do not change the curve outside a neighbourhood of $\Delta$, and the results of doing the two operations in different order inside the neighbourhood are homotopic as can be seen in Figure~\ref{fig: commute}.
\end{proof}

It follows from Lemma~\ref{lem: smoothing invariants htpy class} that the existence of taut smoothings is an invariant of freely homotopic taut curves, and moreover, that if a smoothing at a vertex $v$ of $\alpha$ is taut then the corresponding smoothing at the image $v'$ in any other taut configuration $\alpha'$ is also taut.

 \begin{figure}
\centering
\includegraphics[height=1.4in]{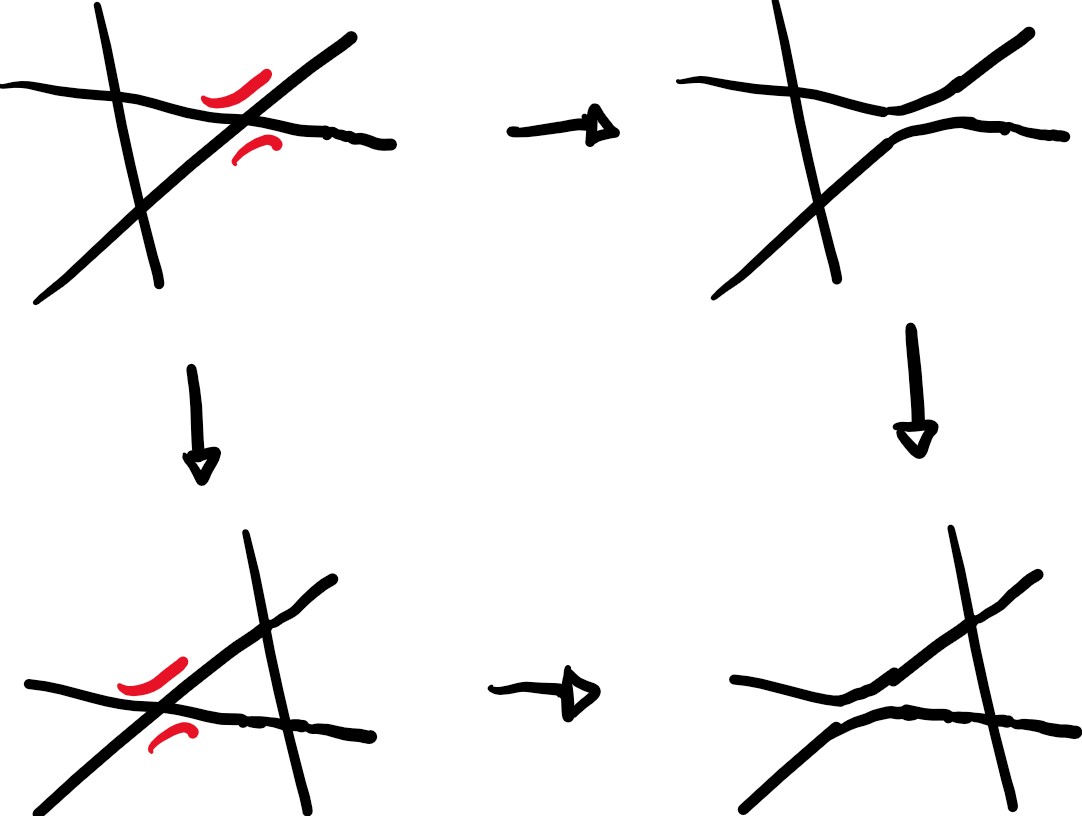}
\caption{Smoothing a double point of a small triangle commutes with performing a triangle move on the same triangle. }
\label{fig: commute}
\end{figure}

\begin{remark}
Lemma~\ref{lem: smoothing invariants htpy class} does not hold for more than $1$ smoothing: i.e., if two or more vertices of  taut freely homotopic 1-manifolds are smoothed simultaneously, then the resulting 1-manifolds may not be freely homotopic. An example is given in Figure~\ref{fig: no commute}.
\end{remark}

 \begin{figure}
\centering
\includegraphics[height=1.7in]{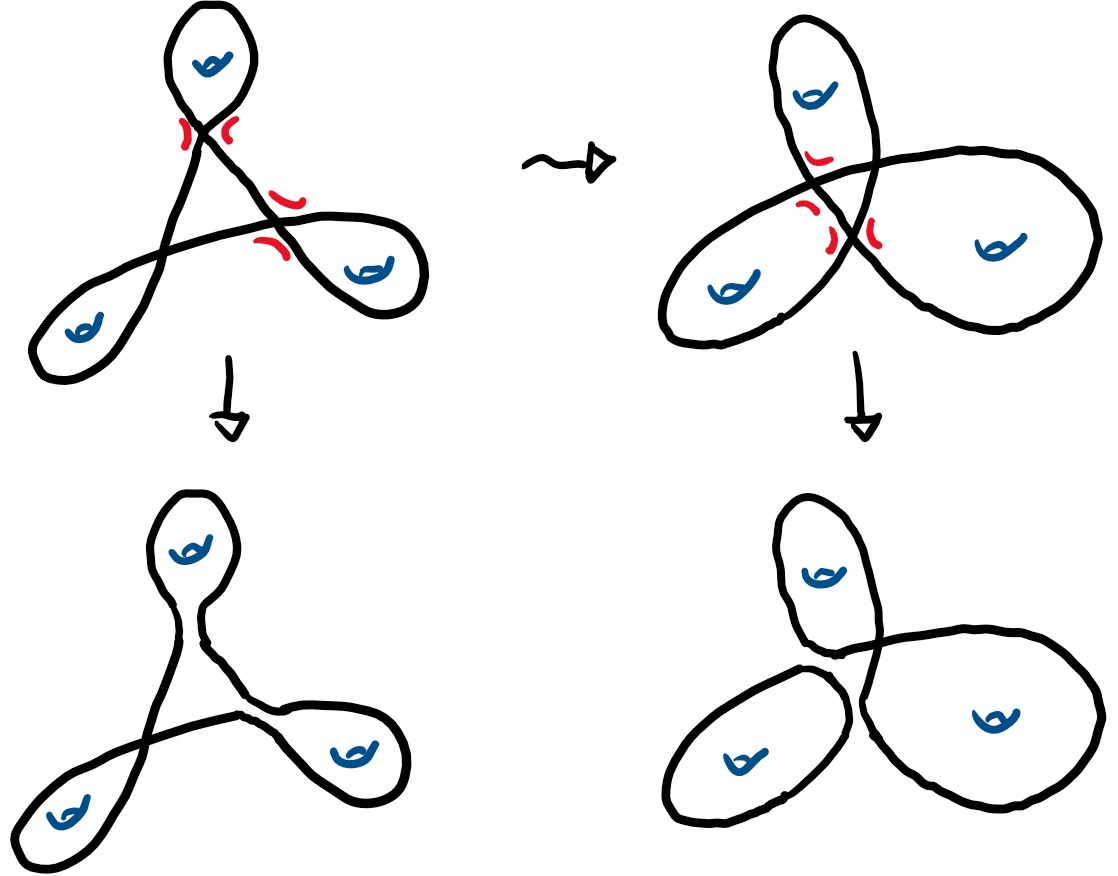}
\caption{Smoothings at two points of  freely homotopic curves which yield non-freely homotopic results.}
\label{fig: no commute}
\end{figure}

In general, non-primitive curves cannot be smoothed to taut curves (see Figure~\ref{fig onlyonedis}). In fact, the only taut smoothings of a curve representing a power are those that disconnect it into curves that represent smaller powers.

  \begin{figure}[h]
\centering
\includegraphics[height=1.1in]{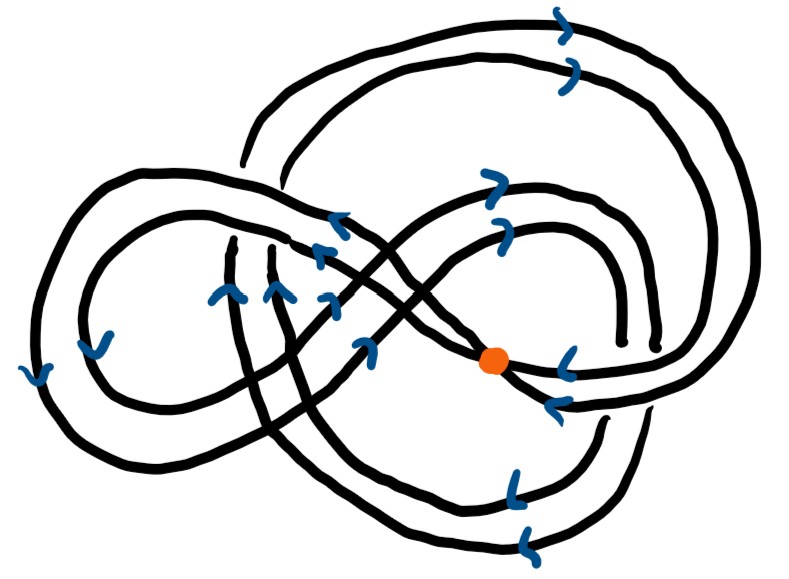}
\caption{A taut curve representing a square in $\pi_1\Sigma$ has only one taut smoothing, which disconnects the curve.}
\label{fig onlyonedis}
\end{figure}

\subsection{Oriented and unoriented corners}\label{subsec: repel}

Viewing the image of an oriented 1-manifold $\alpha$ as an embedded, 4-regular graph in $\Sigma$ whose edges are labelled with orientations, a  \emph{corner} $c$ of $\alpha$ is a pair of  half-edges  of  $\alpha$  that correspond to adjacent vertices in  the link of $p$ in $\Sigma$, where $p$ is a double point of $\alpha$. 
We introduce some necessary terminology:

\begin{definition}\label{def: corner orientations}
A corner $c$ is \emph{unoriented} if the two arcs forming $c$ point towards $p$, or both point away from $p$; otherwise $c$ is \emph{oriented}. 
\end{definition}

Since we only consider curves in general position, around each intersection point of a 1-manifold there are  exactly $4$ corners, which naturally fall into $2$ opposite pairs: one pair of unoriented corners and one pair of oriented corners. So, locally, if an arc ``sees'' an oriented corner to one side, then it sees an unoriented corner to the other side.

A smoothing is \emph{oriented} if it is determined by a pair of oriented corners and it is  \emph{unoriented} otherwise.
The ensuing observation
will be used many times throughout the rest of the paper.

\begin{remark}\label{rmk:orientations and connectedness}
For  1-manifolds, smoothings do not change the number of arc components.
The number of curve components goes up by 1 for oriented smoothings on a single component, goes down by 1 for all smoothings between two components, and remains the same in all other cases.
\end{remark}

We also note that, in view of the above, the result of a smoothing does not necessarily come with an induced orientation. Indeed, for a smoothing on a curve or arc, the result only inherits an orientation if the smoothing was oriented; for a smoothing between distinct connected components, whether the connected components of the result inherit  an orientation depends on the choice of orientations on the original 1-manifold.

\section{Classifying pathologies and chasing maximal hyperbolic angles}\label{sec: classifying}

In this section we classify the small monogons and  bigons that can arise from a smoothing of a taut 1-manifold.
Additionally, we give an argument using maximal angles to show that most 1-manifolds in an orientable surface with a pair of intersecting components
admit taut smoothings, and explain why this argument fails for a single curve or arc. These results and observations serve to motivate our approach in later sections by exposing some of the difficulties intrinsic to Question~\ref{qn:1 intro}.

We start with a few definitions.

Recall that if $\alpha$ is an essential 1-manifold on a surface $\Sigma$, then its preimage $\tilde \alpha$ in $\widetilde \Sigma$ consists of immersed topological lines and/or arcs. 

An \emph{n-gon} of $\tilde \alpha$ is a simple closed curve in $\widetilde \Sigma$ made of $n$ subarcs of the lines above $\alpha$. Thus, a $1$-gon, also called a \emph{monogon}, is a simple loop of an immersed line $\tilde \alpha$, and a $2$-gon, also called a \emph{bigon}, is a pair of simple arcs in $\tilde \alpha$ that meet only at their endpoints. We will also refer to $3$-gons and $4$-gons as \emph{triangles} and \emph{quadrilaterals}, respectively.
When discussing arcs, we  must also consider \emph{$\partial$-n-gons}, which are simple closed curves in $\widetilde \Sigma$ made of $n$ subarcs of the lines and arcs above of $\alpha$ and 1 arc of $\partial \widetilde \Sigma$. A $\partial$-2-gon is called a \emph{half-bigon}.

An $n$-gon of $\alpha$ is the projection of an $n$-gon of $\tilde \alpha$. Notice that
 n-gons of $\alpha$ are generally not embedded: they may have self-intersections and they may have overlapping arcs.

It follows from~\cite{FHS82} that a primitive 1-manifold $\alpha$ on a surface $\Sigma$ is taut with fixed endpoints if and only if $\alpha$ does not contain monogons (so the lines of $\tilde \alpha$ are embedded) nor bigons (so the lines of $\tilde \alpha$ cross each other in at most one point). Similarly, $\alpha$ is taut with free endpoints if and only if $\alpha$ does not contain monogons, bigons, or half-bigons.


\begin{definition}\label{def:small ngons}
A \emph{small} $n$-gon of $\alpha$ is the projection of an $n$-gon of $\tilde \alpha$ that is injective except possibly at a finite set of double points, which are not vertices of the $n$-gon.
\end{definition}

In other words, the sides of a small $n$-gon do not overlap and its corners are not identified in $\alpha$.
The relevance of small monogons, bigons, and triangles is that they give rise to equivariant moves on $\tilde \alpha$ that in some cases correspond to generalised Reidemeister moves on $\alpha$.

Much of this paper hinges on the following result:

\begin{proposition}\label{prop:small}~\cite{HassScott94}
    In an orientable surface $\Sigma$:
 \begin{enumerate}
    \item     If a curve is not taut then it has a small monogon or bigon.
    \item     If an arc is not taut with fixed endpoints then it has a small monogon or bigon.
    \item    If an arc is not taut with free endpoints then it has a small monogon,  bigon, or  half-bigon.   
 \end{enumerate}   
\end{proposition}

\begin{remark}
   Proposition~\ref{prop:small}  does not hold for multi-curves, for non-primitive curves, nor for curves or arcs on non-orientable surfaces. See Figure~\ref{fig: nosmall}.
\end{remark}

 \begin{figure}[h]
\centering
\includegraphics[height=1in]{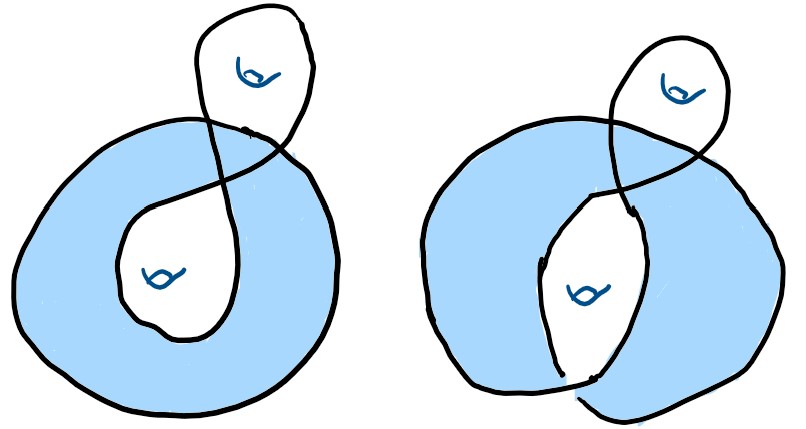}
\caption{A non-taut multi-curve (left) and a non-taut curve in a non-orientable surface (right) without small monogons or bigons
(the ``missing intersection'' is a twist in a Möbius band).}
\label{fig: nosmall}
\end{figure}

\medskip

If $D$ is a polygon of $\alpha$ in $\Sigma$, its \emph{combinatorial area} $\area(D)$ is the number of regions of $\widetilde \Sigma-\widetilde \alpha$ inside $\widetilde D$. 

The following lemma, due to the second author and P. Scott, is proven in~\cite{NCS24+};  we include an outline of the proof for completeness.

\begin{lemma}\label{lem: minimalarea}
    If a curve or arc $\alpha$ in an orientable surface is not taut and  does not contain any monogons, then any  bigon of $\alpha$ with minimal combinatorial area is small.
\end{lemma}

\begin{remark}\label{rk:minimal area is innermost}
Note that a bigon with minimal combinatorial area is \emph{innermost} in the sense that it does not contain other bigons, but an innermost bigon may not have minimal combinatorial area and it may not be small.
\end{remark}

\begin{proof}
Let $B$ be a bigon of $\widetilde  \alpha$ with minimal combinatorial area. Then $B$ is innermost, and if $B$ is not small, then the two sides $s$ and $s'$ of $B$ overlap in  subarcs of $\alpha$ (they cannot overlap only at the endpoints because $\alpha$ is primitive).

Let $c$ and $c'$ be maximal subarcs of the sides with the same image in $\alpha$, and let $b$ and $d$ be their complementary arcs in $s$ and $s'$. These arcs inherit orientations from the orientation of $\alpha$.
The possible configurations of these arcs  depend on the relative orientations of $s$ and $s´$ given by the orientation of $\alpha$ and are illustrated in Figure~\ref{fig: twobigons}.

\begin{figure}[h]
\centering
\includegraphics[height=.8in]{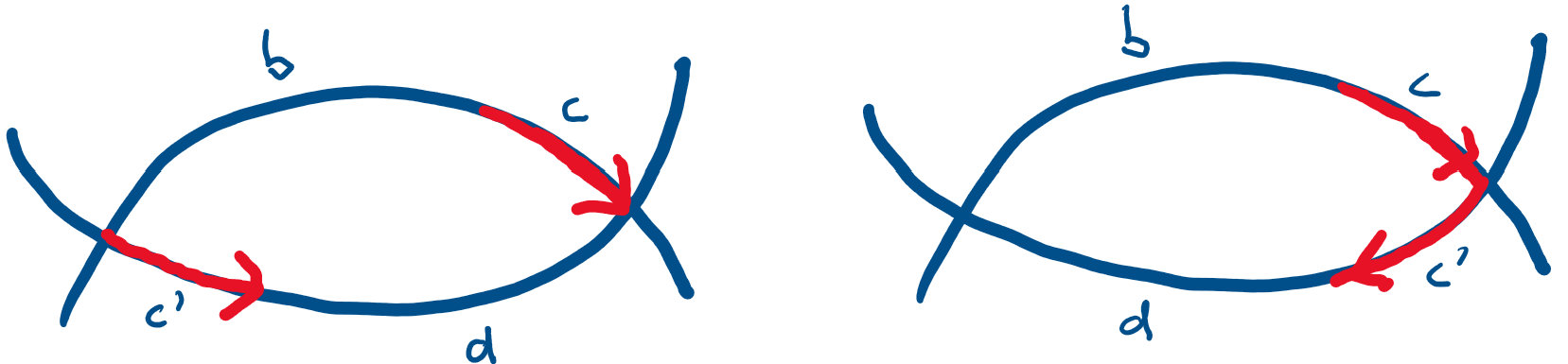}
\caption{Possible configurations of the maximal subarcs with the same image in $\alpha$.}
\label{fig: twobigons}
\end{figure}

If $s$ and $s'$ have opposite orientations then the covering translation taking $c$ to $c'$  has a fixed point, so we may assume that $s$ and $s'$  are oriented in the same way.
Now consider the quotient space $A= B / c \sim c'$ obtained by identifying $c$ and $c'$. Then $A$ is an annulus, and the image of $B$ in $\Sigma$ is a quotient of $A$. We give the images of the arcs $b$, $c$ and $d$ in $A$ the same names. Now
$bcd$ is an arc without corners in $A$. Let $a$ and $e$ be the arcs in the preimage of $\alpha$ in $A$ that lie before and after $bcd$
so that $abcde$ projects to an arc without corners on $\alpha$, as in Figure~\ref{fig: annulusbigon}.

\begin{figure}[h]
\centering
\includegraphics[height=1in]{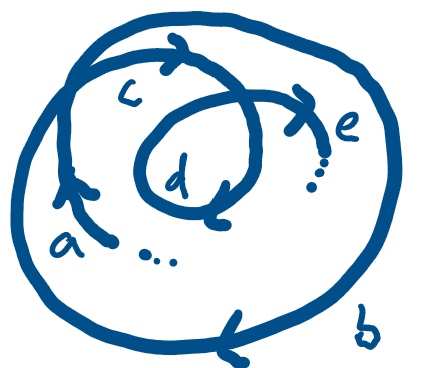}
\caption{The arc $abcde$.}
\label{fig: annulusbigon}
\end{figure}

We will show that the arc $abcde$ contains a bigon in $A$. Its location depends on whether or not $a$ and $e$ are embedded and how they intersect each other (the arc $c$ will not be used in the argument).

Observe that the arc $abcde$  cannot form monogons in $A$ by hypothesis (otherwise $\alpha$ would have them in $\Sigma$), so every simple subloop of $abcde$ goes once around $A$.
If $e$ is not embedded, let $e_1$ be the first subloop of $e$ that closes up after $d$, and let $e_0$ be the subarc of e before $e_1$.
If $e_1$ winds around $A$ in the same direction as $d$ then the arcs $de_0$ and $e_0e_1$ form a bigon in $A$.
If $e_1$ winds around $A$ in the opposite direction, then $e_0$ and $de_0e_1$ form a bigon in $A$.
See Figure~\ref{fig: notembedded annulus}. The same applies if $a$ is not embedded.

\begin{figure}[h]
\centering
\includegraphics[height=1in]{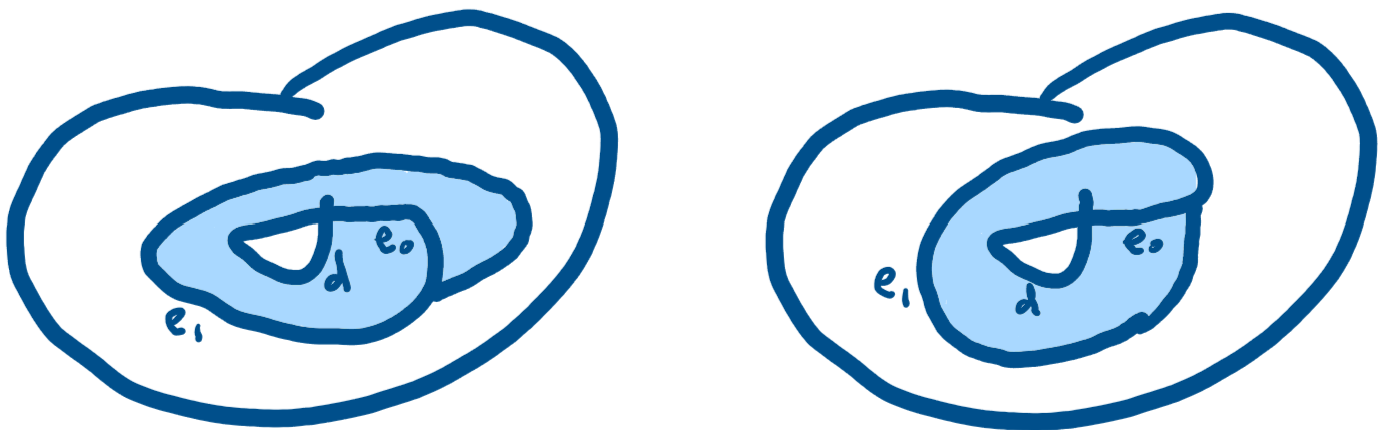}
\caption{Contradictions that arise if either $a$ or $e$ are not embedded.}
\label{fig: notembedded annulus}
\end{figure}

We can therefore assume that $a$ and $e$ are embedded. The arcs $a$, $c$ and $e$ can wind around $A$, and they may intersect each other multiple times, but as $e$ is simple, the arc $e$ does not wind around $A$.
We have four possible configurations depending on the position of $a$ (that of $c$ will be irrelevant):

 \begin{enumerate}
    
\item  If $a$ and $e$ are disjoint then $a$ and $de$ form a bigon in $A$, as otherwise $ab$ would form a monogon in $A$. See Figure~\ref{fig: 4cases},  leftmost and second to left.

\item If $e$ crosses $a$ in both directions then a subarc of $e$ between two consecutive crossings in opposite directions and the subarc of $a$ with the same endpoints bounds a bigon in $A$.

\item  If $a$ crosses $e$ always in the same direction as $b$, and $a_0$ is the initial subarc of $a$ that ends at the first intersection with $e$,
then $a_0$ and $de$ form a bigon. See Figure~\ref{fig: 4cases},  second to right.

\item If $a$ crosses $e$ always in the opposite direction as $b$, then the
and $a_1$ is the final subarc of $a$ that starts at the last intersection with $e$, then $a_1b$ and $e$ form a bigon in $A$.
then $a_0$ and $de$ form a bigon. See Figure~\ref{fig: 4cases},  rightmost.
\end{enumerate}

\begin{figure}[h]
\centering
\includegraphics[height=.8in]{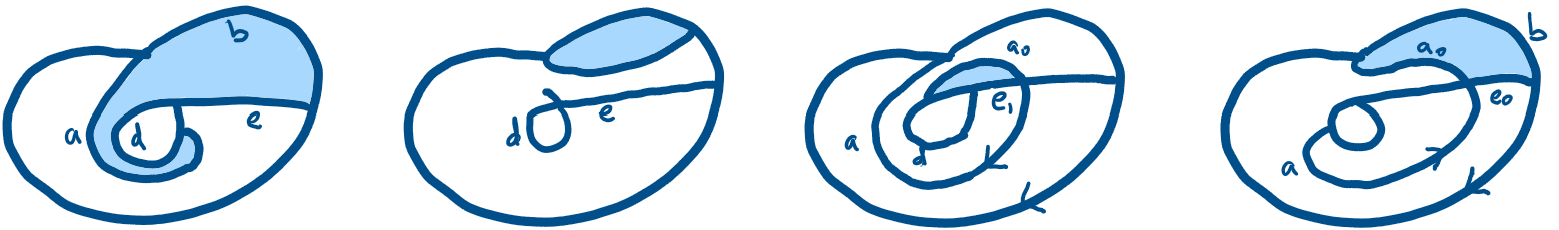}
\caption{The four possible configurations.}
\label{fig: 4cases}
\end{figure}
\end{proof}

\begin{definition}
    A quadrilateral of $\alpha$ with non-overlapping sides and two corners on the same vertex is a \emph{trapezium} if these corners are adjacent in the quadrilateral, and is a \emph{rhombus} if they are not adjacent.
\end{definition}

The next lemma will be used repeatedly throughout the paper.

\begin{lemma}\label{prop:pathologies}
    If $\alpha$ is a taut 1-manifold, then the small monogons and bigons that can arise in a smoothing of $\alpha$ come from triangles, rhombi or trapeziums of $\alpha$.
\end{lemma}

\begin{proof}
    Let $\alpha'$ be the result of smoothing $\alpha$ at a vertex $v$. If $\alpha'$ has a small monogon or bigon, it comes from an n-gon $P$ of $\alpha$. Since $\alpha$ is taut, $P$ has at least 3 sides and contains at least one corner of $v$; and since the monogon or bigon is small, it has no overlapping arcs, so $P$ uses each corner of $v$ at most once. Therefore $P$ contains either one corner or the two opposite corners of $v$, and these are not corners of the monogon or bigon of $\alpha'$. Thus, if $P$ becomes a monogon after the smoothing, then $P$ is a triangle with two corners at $v$, and if $P$ becomes a bigon after the smoothing, then $P$ is either a triangle one corner of $v$, or a quadrilateral with two corners at $v$.
\end{proof}

We emphasise a special case of Lemma~\ref{prop:pathologies}:

\begin{remark}\label{rmk:connectedsmooth4cases}
    If a smoothing $\alpha'$ connects two distinct components of a taut 1-manifold $\alpha$, or does not disconnect a connected 1-manifold, and if $\alpha'$ is not taut, then  it must contain a small monogon, bigon, or half-bigon.
\end{remark}

Remark~\ref{rmk:connectedsmooth4cases} does not hold  for a smoothing that increases the number of connected components of $\alpha$, as distinct components may form bigons without forming small bigons.

The reverse of the previous observation is also useful:

\begin{remark}\label{rmk:connectedsmooth4casesreverse}
    A smoothing of $\alpha$ at a corner of a triangle, or at the touching corners of a trapezium or a rhombus cannot be taut.
\end{remark}

 \begin{figure}[h]
\centering
\includegraphics[height=1in]{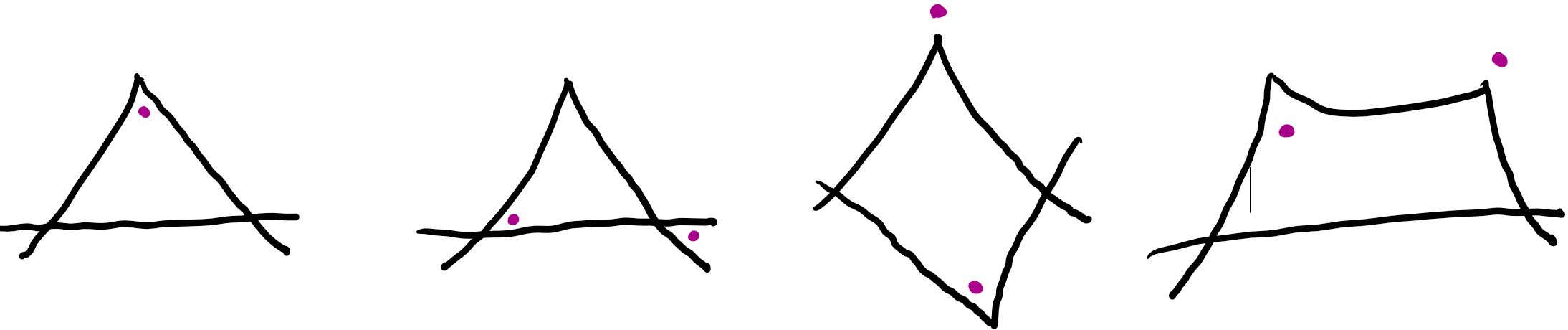}
\caption{The $4$ kinds of small monogons and bigons.}
\label{fig: smalls}
\end{figure}

All the pathologies in Lemma~\ref{prop:pathologies} can occur in taut curves, and in fact, there exist taut curves exhibiting any combination of them; Figure~\ref{fig: onlyoneforarcs} is an example.
As is the case for the curve in Figure~\ref{example maximal}, it can even happen that every self-intersection point is a vertex of a triangle, trapezium or rhombus.

The following simple corollary illustrates the applicability of Lemma~\ref{prop:pathologies}.

\begin{corollary}\label{cor:maincurvessmooth}
     Let $\Sigma$ be an orientable surface. If $\alpha$ is a taut curve on $\Sigma$ that admits a decomposition as $\alpha=\beta \gamma$ where $\beta$ and $\gamma$ intersect only at their endpoints, and no triangle of $\alpha$ has $p=\gamma \cap \beta$ as a vertex, then the smoothing at $p$ that separates $\alpha$ is taut.   
\end{corollary}

In particular, if $\gamma$ is simple, the smoothing produces a taut curve with one intersection less than $\alpha$ (after deleting a simple curve component).

\begin{proof}
   The smoothing at $p$ that separates $\alpha$ produces two disjoint curves $\beta'$ and $\gamma'$. If $\beta'$ has a monogon or bigon then by Proposition~\ref{prop:small} it must have a small monogon or bigon, and by Lemma~\ref{prop:pathologies} this must come from a triangle, a trapezium, or a rhombus of $\alpha$. But by hypothesis, there are no triangles with vertex $p$ in $\alpha$, and there cannot be any rhombi or trapeziums with vertex $p$ because they would have to use two opposite corners of $\alpha$, and $\gamma$ does not intersect any subarc of $\beta$. For the same reason $\gamma$ cannot have monogons or bigons. There are also no bigons between $\beta$ and $\gamma$ because both curves are disjoint, so the proof is complete. 
\end{proof}

\subsection{Using hyperbolic geometry}

Since the existence of taut smoothings for a taut 1-manifold $\alpha$ depends only on the free homotopy class of $\alpha$, a plausible strategy for finding taut smoothings is to choose taut representatives with special properties. If the components of $\alpha$ are non-parallel, then $\alpha$ is freely homotopic to a geodesic 1-manifold
for some hyperbolic metric on $\Sigma$, and we can choose a smoothing in a corner with a maximal hyperbolic angle. We shall see that this works in many cases, but not always.

\begin{theorem}\label{prop:taut pair}
  Let $\Sigma$ be an orientable surface with $\chi(\Sigma)<0$. If a primitive, taut 1-manifold $\alpha$ in $\Sigma$ has at least $2$ non-disjoint components, then $\alpha$ has a smoothing that is taut fixing its endpoints.
\end{theorem}

\begin{remark}
    Theorem~\ref{prop:taut pair} does not hold on a torus: for a taut manifold consisting of two simple curves $\alpha$ an $\beta$ that intersect more than once, the smoothings of $\alpha \cup \beta$ produce curves that are non-simple, and for most choices of $\alpha$ an $\beta$ they are primitive. Thus, they cannot be   taut. 
\end{remark}

\begin{proof}
    For simplicity, we will give a detailed proof when $\alpha$ is a collection of curves, and at the end we indicate what to do if $\alpha$ has arc components.

    If $\alpha$ is a taut collection of primitive, non-homotopic curves, then $\alpha$ can be homotoped to a collection $\alpha_h$ of (primitive, non-homotopic) hyperbolic geodesics,
    and by Lemma~\ref{lem: smoothing invariants htpy class}
    the taut smoothings of $\alpha_h$ correspond to the taut smoothings of $\alpha$. We claim that a smoothing $\alpha'$ of $\alpha_h$ at an unoriented corner forming a maximal angle between two distinct components is taut.

    To prove that $\alpha'$ is taut we need to show that the component $\alpha_1'$ that contains the 2 corners of the smoothing self-intersects minimally and that $\alpha_1'$ intersects every other component of $\alpha'$ minimally.
    
    If $\alpha_1'$ does not have minimal self-intersections then by Proposition~\ref{prop:small} it must have a small monogon or bigon. 
    If the monogon or bigon comes from a triangle $T$, then a corner of the smoothing corresponds to an interior angle of $T$, but as $T$ is a hyperbolic triangle the other two exterior angles of $T$ are larger, and at least one of them is the corner of two distinct components of $\alpha_h$, contradicting our initial hypothesis.

    If the bigon comes from a quadrilateral $Q$, then it must be a rhombus rather than a trapezium, since a trapezium has two consecutive vertices that map to the same vertex in $\Sigma$, and as the surface is orientable, the arcs that meet at this vertex correspond to the same component of $\alpha_h$.
    Also, since the surface is orientable, the two corners of $Q$ that do not correspond to the smoothing must be corners of different components of $\alpha_h$, so their complementary corners are also corners of different components. Since $Q$ is a hyperbolic quadrilateral, the angle at one of these corners must be larger than the angle at $c$, contradicting again our initial hypothesis. See Figure~\ref{fig: anglecases}, second row.

    The previous argument shows that $\alpha_1'$ is taut and also that $\alpha_1'$ does not form small bigons with other components of $\alpha'$. We still need to show that $\alpha_1'$ cannot form large bigons with any other component of $\alpha'$.
    If $\alpha_1'$ and another component $\alpha_2$ form a bigon, then it lifts to a bigon $B$ in $\widetilde \Sigma$ whose sides lie on two preimages $\widetilde \alpha_1'$ and $\widetilde\alpha_2$. 
    Since $\widetilde\alpha_2$ is a geodesic line, $\alpha_1'$ is made of arcs of two components of $\alpha$, and because $\Sigma$ is orientable, then $\widetilde \alpha_1'$ is a zigzag of geodesic arcs that turn left and right alternately.
    Now an innermost bigon $B$ of this type must be a triangle, as otherwise the side of $B$ on $\widetilde \alpha_1'$ would have a corner that points out of $B$, and the opposite corner would lie on another preimage of $\alpha_1'$ that forms a bigon $B'$ with $\widetilde \alpha_2$ that is contained in $B$.
    But we already showed that a smoothing at a corner with maximal angle cannot form a bigon that comes from a triangle, so we are done.

    If $\alpha$ has some freely homotopic components, let $\alpha_r$ be the reduced 1-manifold that has one curve from each homotopy class in $\alpha$. Then $\alpha_r$ is taut and is homotopic to a collection of (primitive and non-homotopic) hyperbolic geodesics. 
    The previous argument shows that $\alpha_h$ (and therefore $\alpha_r$) has a taut smoothing, provided that $\alpha_h$ (and therefore $\alpha_r$) has some non disjoint components. 
    Assume that this is the case, and let us prove that a corresponding smoothing on $\alpha$ is taut. Since the lines of $\widetilde \alpha$ run parallel to the lines of $\widetilde \alpha_r$ in $\widetilde \Sigma$ (each line of $\widetilde \alpha$  has the same points at infinity as a line of $\widetilde \alpha_r$, and the lines of $\widetilde \alpha$ with the same points at infinity are disjoint) so each corner of $\alpha$ determines a corner of 
    $\alpha_r$, and each n-gon of $\widetilde \alpha$ determines an n-gon of $\widetilde \alpha$.
    Therefore, if a smoothing of $\alpha$ produces a monogon or bigon then the corresponding smoothing  in $\alpha_r$ also produces a monogon or bigon.

    If all the components of $\alpha_r$ are disjoint, two non-disjoint components of $\alpha$ must be homotopic to a single non-simple component of $\alpha_r$. 
    In the corresponding non-simple component of $\alpha_h$, take an unoriented corner $c$ with largest hyperbolic angle, and perform a smoothing of $\alpha$ at a corner between two different components corresponding to $c$. If the result of this smoothing is not taut then it must contain a monogon or bigon. 
    We proved before that if a smoothing between two distinct curves creates monogons of bigons then it must create small ones that come from triangles or rhombi. But the triangles and rhombi of $\alpha$ correspond to triangles and rhombi of $\alpha_r$, which correspond to triangles and rhombi of $\alpha_h$, and a smoothing of $\alpha_h$ at a corner with maximal hyperbolic angle can only create bigons from trapeziums. This contradiction shows that the smoothing of $\alpha$ is taut.

    The same proof works when some or all of the components of $\alpha$ are arcs.
    If $\alpha$ does not contain homotopic curves then $\alpha$ is freely homotopic to a collection $\alpha_h$ of hyperbolic geodesic curves and arcs with the same endpoints as those in $\alpha$, and a smoothing of $\alpha_h$ at an unoriented corner of maximal hyperbolic angle is taut and so the corresponding smoothing of $\alpha$ is taut. If $\alpha$ contains some homotopic curves, consider a collection $\alpha_r$ that has all the arcs of $\alpha$ and one component of each homotopy class of curves, so $\alpha_r$ is freely homotopic to a collection $\alpha_h$ of hyperbolic geodesic curves and arcs.

    Then, as before, we can do a smoothing at a corner between two non-homotopic components of $\alpha$  corresponding to a corner with maximal hyperbolic angle between two components of $\alpha_h$. If all non-homotopic components of $\alpha$ are disjoint, then two of them are homotopic and non-simple, and we can do a smoothing between them corresponding to a corner with maximal hyperbolic angle in the corresponding component of $\alpha_h$.
    \end{proof}

 \begin{figure}[h]
\centering
\includegraphics[height=2.4in]{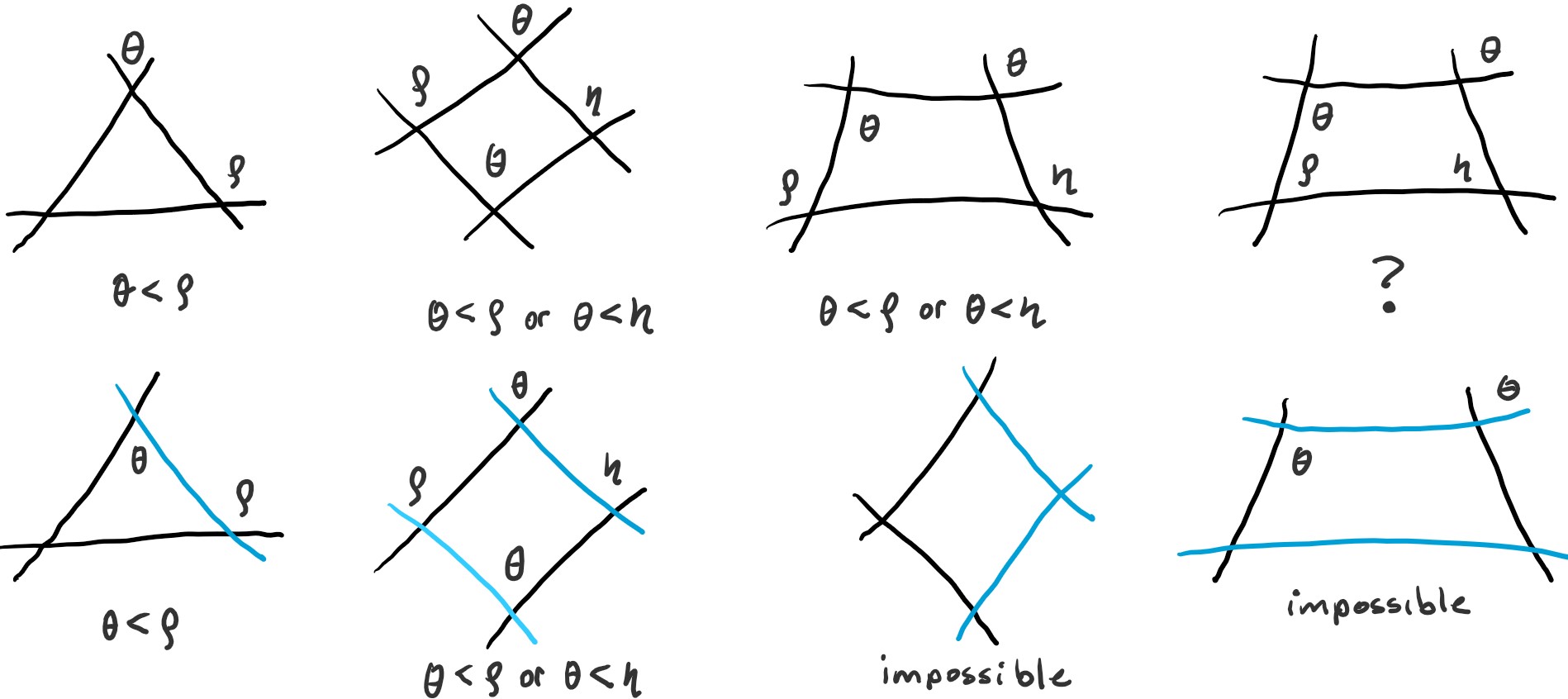}
\caption{Possible and impossible scenarios in the proof of Theorem~\ref{prop:taut pair}.}
\label{fig: anglecases}
\end{figure}

The reader might wonder why this proof fails in the case of a single curve or arc, so let us point out the problem. If $\alpha$ has a single component, we may try to mimic the proof of Theorem~\ref{prop:taut pair}, and choose an unoriented smoothing with maximal hyperbolic angle, so the result has a single component $\alpha'$. If $\alpha'$ is not taut, it must contain a small monogon or bigon $P$ by Proposition~\ref{prop:small}.
Up to reorientation of $\alpha$, and as $\Sigma$ is oriented, there are exactly 5 possibilities for $P$, all shown in the top row of Figure~\ref{fig: anglecases} (really there are $2$ cases corresponding to a triangle -- for one of them two of the internal angles will be equal, but this does not affect the argument).

Since the triangles and quadrilaterals are hyperbolic, the sum of the interior angles of $P$ is either $\pi$ or $2\pi$, so one of the exterior angles is larger than the one chosen for the smoothing, and in the first $4$ cases it corresponds to an unoriented corner, contradicting our choice. But in the last case (a trapezium with 4 unoriented corners) the exterior angle that can be shown to be larger then the original angle corresponds to an oriented corner, so a priori there is no contradiction to the choice of a corner with maximal unoriented angle.
Thus, the smoothings with maximal unoriented angle may not be taut. 

Nevertheless, the discussion above recovers the following corollary:

\begin{corollary}\label{cor:nomomogons}
    Let $\Sigma$ be an orientable surface with $\chi(\Sigma)<0$.   If $\alpha$ is a 1-manifold in $\Sigma$ with  primitive curve components  and $\alpha$ is taut fixing its endpoints, then $\alpha$ has a smoothing that has no monogons. 
\end{corollary} 

Corollary~\ref{cor:nomomogons} was previously obtained by D. Thurston, who proved it without the primitivity assumption, and used it in~\cite{ThurstonD14} to show that the \emph{bracelet basis} is a  positive bases for the skein algebra $Sk(\Sigma)$.

 If we pick the maximal angle without regard for orientations, the preceding analysis allows us to deduce that there is a smoothing with taut components (but note that, in general, we cannot say anything about how the components intersect each other).

\begin{corollary}\label{cor:tautcomponents}
    Let $\Sigma$ be an orientable surface with $\chi(\Sigma)<0$.   If $\alpha$ is a primitive taut curve, or an arc that is taut fixing its endpoints, then $\alpha$ has a smoothing for which each connected component is taut. 
\end{corollary}

To finish this section, we show that there are examples of homotopy classes of curves for which the approach in Theorem~\ref{prop:taut pair} will definitely fail. 

\begin{proposition}\label{prop:nomaxmetric}
There exists a taut curve $\alpha$ on a surface $\Sigma$ with the following properties

\begin{enumerate}

    \item $\alpha$ can be smoothed to a taut curve,
    
    \item every unoriented corner of $\alpha$ is the corner of a quadrilateral,

    \item the unoriented corners of $\alpha$ that correspond to taut smoothings and those that do not are indistinguishable in the universal cover of $\Sigma$, and

    \item there are continuously many hyperbolic metrics on $\Sigma$ for which the smoothings of $\alpha$ that produce taut curves are not in the corners with largest unoriented angles.
\end{enumerate}   
\end{proposition}

\begin{proof} 
Let $\alpha$ be the curve on a sphere with four holes $\Sigma$ shown in Figure~\ref{example maximal}. The unoriented smoothings at the corners marked with $\theta$ are taut, those at the corners marked with $\rho$ are not. However all the corners look the same in $\widetilde \Sigma$.

Note that $\alpha$ cuts $\Sigma$ into a hexagon, 3 quadrilaterals, and 4 annuli. The taut smoothings are at the corners marked $\theta$.
We can give $\Sigma$ different hyperbolic structures in which $\alpha$ is a geodesic by taking a hyperbolic hexagon $H$ with equal sides and equal angles $\theta$ and $\rho$ at the marked corners for some $\pi/2<\theta, \rho<2/3\pi$ and taking for each quadrilateral a hyperbolic trapezium $T$ with 2 sides matching those of $H$ and interior angles $\pi-\theta$ at the top and and $\pi-\rho$ at the base. It is then easy to choose the shapes of the annuli to give $\Sigma$ a hyperbolic structure with geodesic boundary (or a complete hyperbolic metric with 4 cusps and finite area).
The angles $\theta$ and $\rho$ can be chosen independently, and if $\pi/2<\theta<\rho<2/3\pi$ the taut smoothings are not in the corners with largest hyperbolic angles.
\end{proof}

 \begin{figure}[h]
\centering
\includegraphics[height=2in]{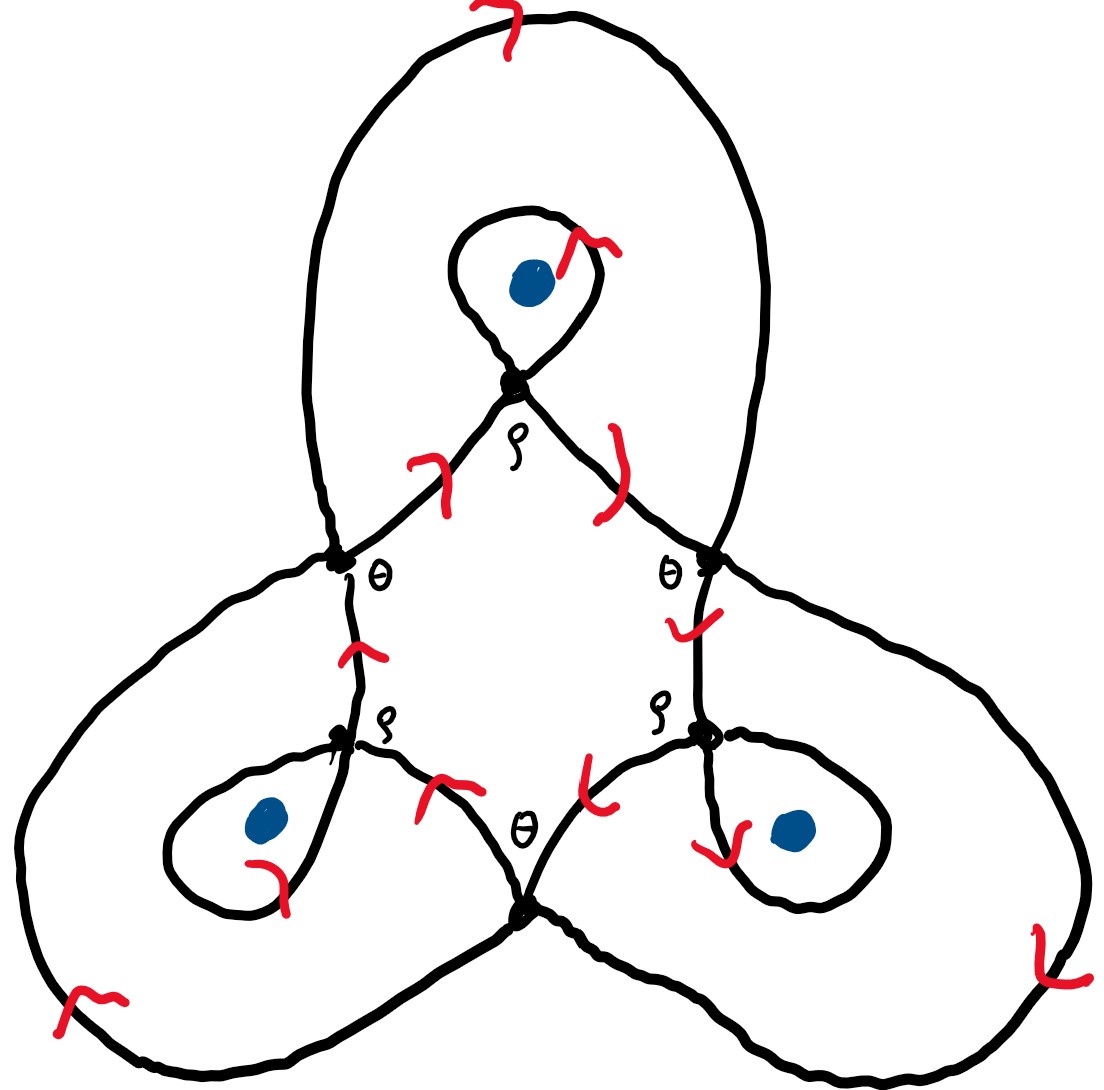}
\caption{A curve that can be endowed with hyperbolic metrics where the maximal angle does not produce a taut smoothing.}
\label{example maximal}
\end{figure}

\section{Finding taut smoothings directly}\label{sec:taut smoothing}

In order to find taut smoothings for some families of curves and arcs on a surface $\Sigma$, we will need to perform regular isotopies that change their configurations, 
and we need a combinatorial measure of angles that is invariant under these moves.

If $\alpha$ is a taut multi-curve on a surface $\Sigma$, its preimage $\widetilde \alpha$ in the universal cover $\widetilde \Sigma$ consists of embedded lines.
If the components of $\alpha$ are primitive, the lines cross each other in at most one point, and if $\chi(\Sigma)<0$ then each pair of lines is crossed only by a finite number of lines. 
One can use these properties to define a combinatorial measure for the angles at the corners of $\alpha$. 
If $\alpha$ is a taut 1-manifold on a surface with $\chi(\Sigma) < 0$ then its preimage in $\widetilde \Sigma$ consists of a mix of embedded lines and embedded arcs with similar properties, so an analogous definition also works in this case. 

\begin{definition}[Combinatorial angle]
 Let $l_1$ and $l_2$ be two lines of $\widetilde \alpha$ that intersect at a point $p$, and let $l_3,l_4,...,l_n$ be lines of $\widetilde \alpha$ that cross both $l_1$ and $l_2$.
 The points at infinity of $l_1,l_2,l_3,...,l_n$ cut the circle at infinity into $2n$ arcs, which belong to one of the four sectors at infinity determined the 4 corners of $p$. Define the combinatorial angle of a corner $c$ of $p$ as the number of arcs at infinity that lie in the sector determined by $c$, divided by the total number of arcs.
 The \emph{combinatorial angle of a corner of $\alpha$} is the combinatorial angle of a corresponding corner of $\widetilde \alpha$.
\end{definition}

It follows from the definition that opposite corners at a vertex of $\alpha$ have the same combinatorial angle, and that the combinatorial angles of  pairs of adjacent corners add up to $\frac{1}{2}$.

\begin{lemma}\label{lem:combinatorial angle config}
    The combinatorial angles of a taut 1-manifold $\alpha$ on a surface $\Sigma$ only depend on the free homotopy class of $\alpha$.
\end{lemma}

\begin{proof}
    The combinatorial angle between two lines or arcs of $\widetilde \alpha$ is determined by the cyclic ordering of the ends of the lines and arcs that cross them in $\widetilde \Sigma$. These ends are invariant under proper homotopies of $\alpha$ (if the endpoints of the arcs are fixed)
    and the order of the ends is invariant if the endpoints of the arcs move without forming half-bigons. 
\end{proof}

Lemma~\ref{lem:combinatorial angle config} highlights an advantage of working with combinatorial angles rather than fixing a metric and working with hyperbolic angles: when fixing a metric one looses control over the configuration.

Combinatorial angles also have the following useful property:

\begin{lemma}\label{lem:combinatorial angle geometry}
If $\alpha$ is a taut 1-manifold, then the sum of the internal combinatorial angles of every triangle of $\alpha$ is at most $\frac{1}{2}$,
so the internal angle at each vertex is smaller than the external angles at the other two vertices.
\end{lemma}

\begin{proof}
    We give a proof for multi-curves (it is the same for multi-arcs and arbitrary $1$-manifolds, changing lines by arcs or by the appropriate mix of both). 
    
   Take a triangle $\Delta$ of $\widetilde \alpha$. 
   Let $c_1$, $c_2$, $c_3$ be the inner corners of $\Delta$, and let $l_1$, $l_2$, $l_3$ be the lines that contain the opposite sides of $\Delta$.
    Let $L$ be the set of all the other lines of $\widetilde \alpha$ that cross at least two lines of $\Delta$, let $L_i$ be the subset of lines that cross $l_j$ and $l_k$, and $L_i^+$ the subset of these lines whose ends lie in the sectors determined by $c_i$ and its opposite corner. Then $ \angle c_i = \frac{1+|L_{i}^+|}{2(3+L_{i}|)}$.


If all the lines in $L$ cross $l_1$, $l_2$, $l_3$ then $L_i=L_j=L_k=L$ and $|L_{i}^+|+|L_{j}^+|+|L_{k}^+|=|L|$ so $ \angle c_i + \angle c_j + \angle c_k=1/2$

If some line $l'$ of $L$ does not cross $l_i$ then $l'$ lies in $L_i$ but $l'$ does not lie in $L_i^+$, 
so the fraction that gives $\angle c_i$ becomes smaller and the sum of the angles becomes smaller than $\frac{1}{2}$.
\end{proof}

The triangle moves of a taut 1-manifold $\alpha$ on a surface $\Sigma$ are generally defined for triangles of $\alpha$ that are embedded and innermost, as isotopies of $\alpha$ that push a neighbourhood of one side of the triangle across the triangle, flipping it. The resulting 1-manifold $\alpha'$ is taut and does not depend on which side of  the triangle $\Delta$ is being pushed.

If a triangle $\Delta$ is embedded but is crossed by other arcs of $\alpha$, we can do a triangle move that pushes one side of the triangle across $\Delta$, but to preserve tautness we must assume that the opposite corner  $c$ is not the corner of a triangle contained in $\Delta$ (we then say that $\Delta$ is an \emph{innermost triangle with corner $c$}). See Figure~\ref{fig: big1}. In this case,  the result does depend on the side $s$ that is being pushed across $\Delta$.

\begin{figure}[h]
\centering
\includegraphics[height=1.05in]{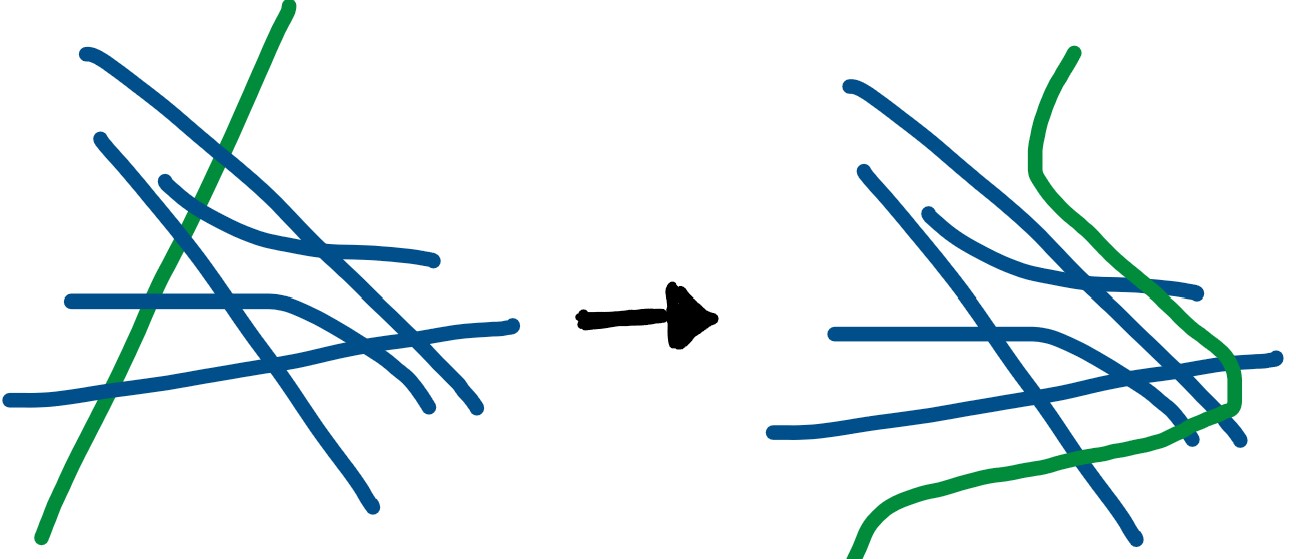}
\caption{Pushing an arc across a big triangle.}
\label{fig: big1}
\end{figure}

If a triangle $\Delta$ is not embedded but  is small (so its sides can be moved independently), we can similarly do a triangle move that pushes one side across $\Delta$, assuming that the opposite corner is not the corner of another triangle contained in $\Delta$. However, the resulting 1-manifold $\alpha'$ may not be taut, and even if $\alpha'$ is taut, the isotopy that transforms $\alpha$ into $\alpha'$ looks complicated, because $s$ may cross $\Delta$ many times. See Figure~\ref{fig: big2}.

\begin{figure}[h]
\centering
\includegraphics[height=1.05in]{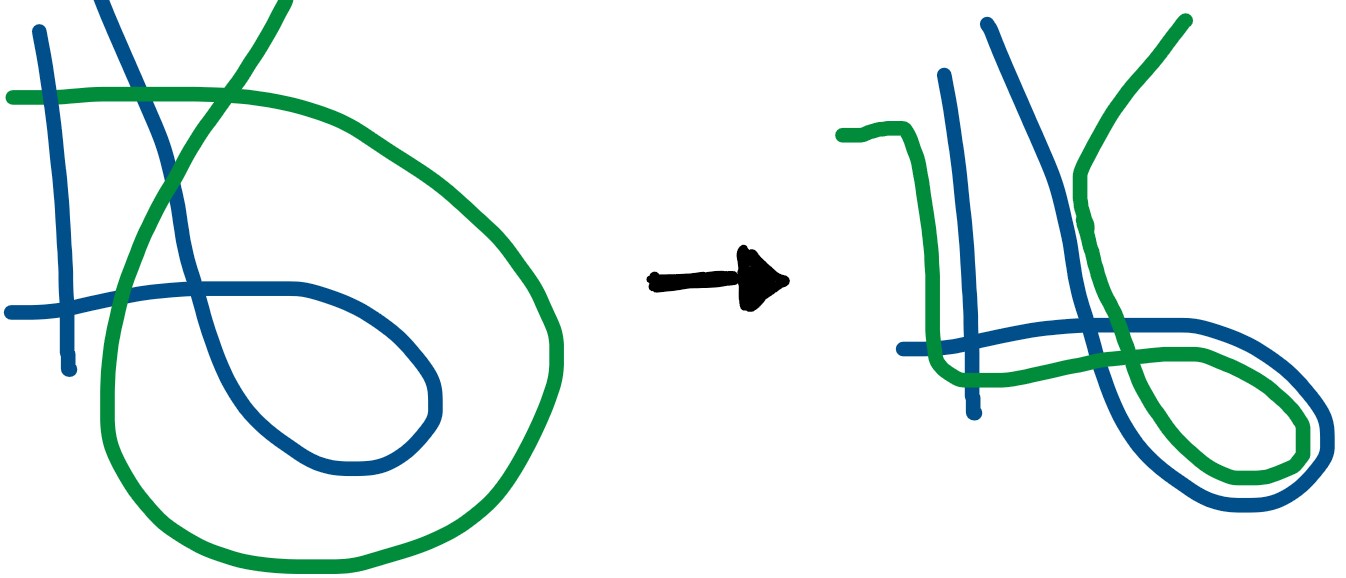}
\caption{Pushing an arc across a non-embedded triangle.}
\label{fig: big2}
\end{figure}

\begin{lemma}\label{lem: moveable triangles} Let $\alpha$ be a taut curve or arc on an orientable surface $\Sigma$. If $c$ is a corner of a small triangle $\Delta$ of $\alpha$, then there is a triangle move of $\alpha$ that pushes the opposite side across $\Delta$ and preserves tautness, or $c$ is the corner of another triangle with smaller area (which may not be contained in $\Delta$). \end{lemma}

\begin{proof}
As $\Delta$ is small, its sides are subarcs of $\alpha$ that do not overlap, so we can push the side $s_1$ opposite to $c$ across $\Delta$ without moving the other two sides. But this move may not preserve tautness, even if $\Delta$ is an innermost triangle with corner $c$. To see why, observe that the preimage $\widetilde \alpha$ of $\alpha$ in the universal cover $\widetilde \Sigma$ is a collection of lines or arcs $\tilde \alpha_i$ that cross each other in at most one point.
The isotopy of $\alpha$ lifts to an isotopy of $\widetilde \alpha$ that pushes all the preimages of $s_1$ across all the preimages of $\Delta$, and these simultaneous moves can create new intersections between the resulting lines or arcs $\tilde \alpha'_i$.

Let  $s_i^k$ be the preimages of $s$ in $\tilde \alpha_i$, let $t_i^k$ be the complementary arcs of $\tilde \alpha_i$, let $\widetilde \Delta_i^k$ be the preimages of $\Delta$ adjacent to $s_1^k$, and let
$\Lambda_i^k= \widetilde \Delta_i^k -s_i^k$.

Replacing all $s_l^k$ by $\Lambda_l^k$  can create more intersections between $\tilde \alpha'_i$ and $\tilde \alpha'_j$ when
$\tilde \alpha_i$ crosses some $\Lambda_j^k$ twice, when 
$\tilde \alpha_i$ crosses $t_j^k$ and $\Lambda_j^k$ (because then it crosses $\Lambda_j^k$ again), when $\Lambda_i^k$ crosses $\Lambda_j^l$ twice, 
or when $\Lambda_i^l$ crosses $t_j^k$ and $\Lambda_j^k$.
If $\tilde \alpha_i$ or any of the lines that contain the sides of $\Lambda_i^k$ cross $\Lambda_j^l$ twice, then $\Delta$ is not innermost with corner $c$, and there is a triangle with smaller area. 

\begin{figure}[h]
\centering
\includegraphics[height=1in]{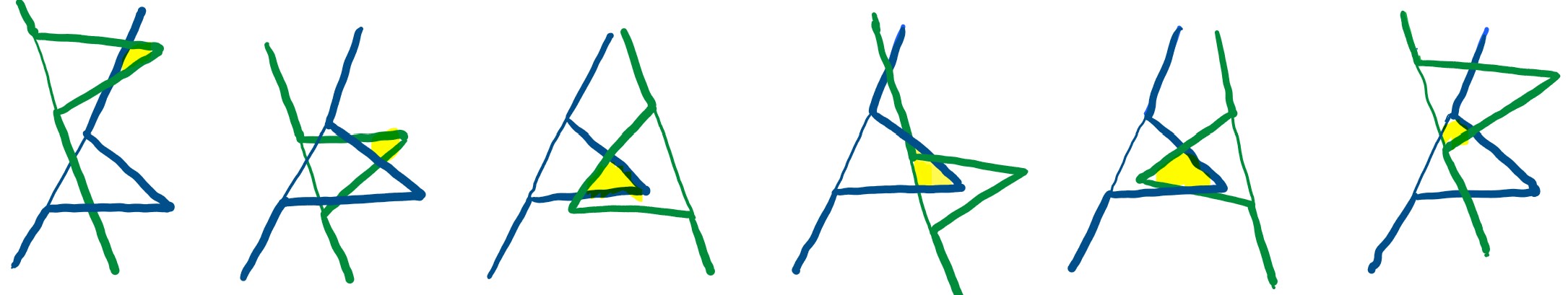}
\caption{Different ways in which a triangle move on $\alpha$ can produce extra intersections on $\widetilde \alpha$.}
\label{triangle1}
\end{figure}

This leaves only two cases: when $\Lambda_i^k$ and $\Lambda_j^l$ cross twice, forming a prism with inner corners $\tilde c_i^k$ and 
$\tilde c_j^k$, as in the second to last configuration in Figure~\ref{triangle1}, and when 
$\widetilde \Delta_i^k$ contains a corner of $\widetilde \Delta_j^k$ different from $\tilde c_i^k$  and vice-versa, as in the rightmost configuration of Figure~\ref{triangle1}.

The first case  can be ruled out because it would mean that the covering translation that sends $\Lambda_i^k$ and $\Lambda_j^l$ has a fixed point. 

In the second case there are two other triangles $\Delta_i'$ and $\Delta_j'$
with corner $\tilde c$ contained in $\widetilde \Delta_i^k \cup \widetilde \Delta_j^l$. 
The triangle $\Delta_i'$ 
is bounded by $\tilde \alpha_i$ and $\widetilde \Lambda_j^l$
and $\Delta_j'$ is bounded by $\tilde \alpha_j$ and $\widetilde \Delta_i^k$ as in Figure~\ref{triangle2}. 
We will show that one of these triangles has combinatorial area smaller than that of $\Delta$, contradicting our initial assumption.
With the complementary regions labelled as in the figure we get:

$$\area (\Delta_i') = \area ( \widetilde \Delta_j^l) - \area  (A) - \area ( B) + \area  (C)$$

$$\area  (\Delta_j') = \area  (\widetilde \Delta_i^k) - \area (C) - \area(D) + \area (B)$$

\begin{figure}[h]
\centering
\includegraphics[height=1.6in]{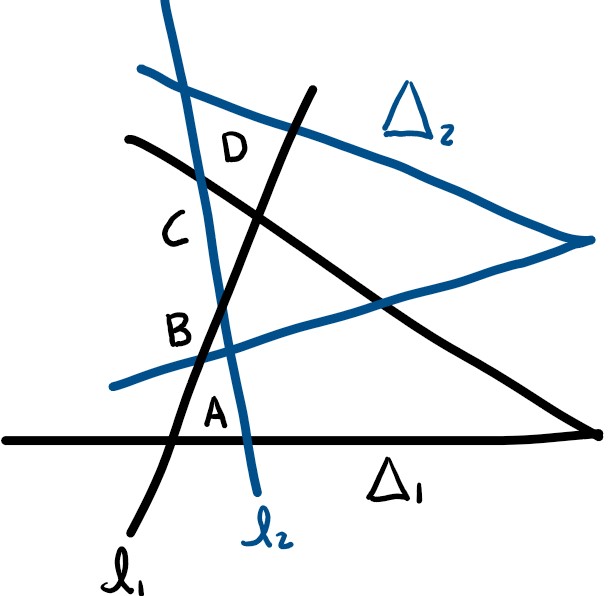}
\caption{$\Delta$ does not have minimal combinatorial area.}
\label{triangle2}
\end{figure}

Therefore 

$$\area (\Delta_i') + \area  (\Delta_j') = \area (\widetilde \Delta_i^k) + \area  (\widetilde \Delta_j^l) - \area  (A) - \area  (B)$$

So either $\area  (\Delta_i') < \area (\Delta)$ or $\area (\Delta_j') < \area (\Delta)$,
contradicting our hypothesis.
\end{proof}

We now state one of the main results of this work:

\begin{theorem}   \label{thm:mainsmooth}
Let $\alpha$ be a taut curve or arc on an orientable surface $\Sigma$ that has an unoriented corner that is not the corner of a quadrilateral or a pentagon. Then $\alpha$ has a connected taut smoothing.
\end{theorem}

\begin{proof}
Among the unoriented corners of $\alpha$ that are not corners of quadrilaterals or pentagons, let $c$ be one with maximal combinatorial angle.
We claim that $c$ is not the corner of a triangle.
If it is, choose a triangle $\Delta$ with corner $c$ that is not contained in a larger triangle (the hypothesis implies that $\chi(\Sigma)<0$, so there are only finitely many triangles).
$\Delta$ has an exterior unoriented corner $c'$ with larger combinatorial angle than $c$. Moreover, $c'$ is not the corner of a quadrilateral or a pentagon, as otherwise $c$ would also be a corner of one, as can be seen by lifting $\Delta$ and the possible quadrilaterals and pentagons with corner $c'$ to the universal covering of $\Sigma$ (as in Figure \ref{triangle3}). This contradicts our choice of $c$.
So we can assume that $c$ is not the corner of any triangles, quadrilaterals or pentagons.

\begin{figure}
\centering
\includegraphics[height=2in]{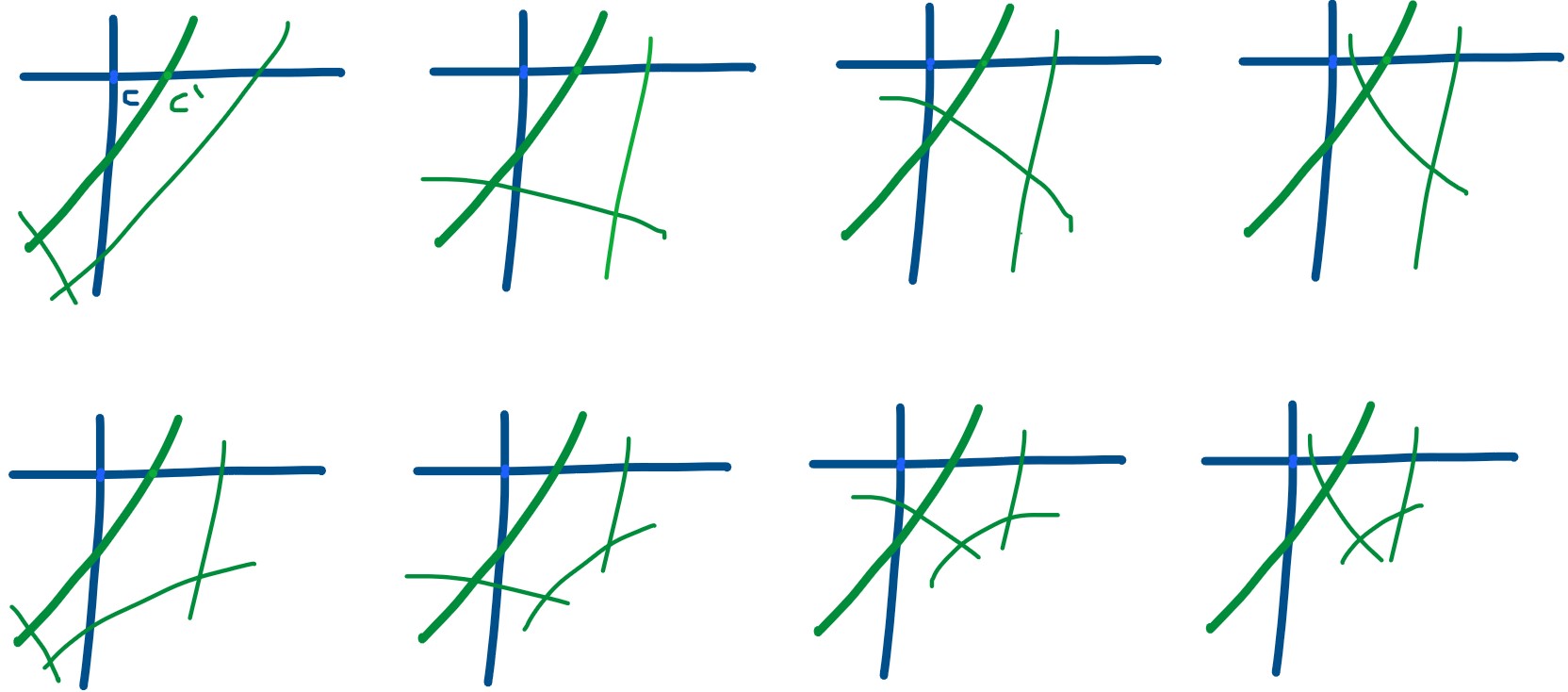}
\caption{
If  $c'$ is the corner of a quadrilateral then $c$ is also the corner of a quadrilateral.}
\label{triangle3}
\end{figure}

Let $c^-$ be the opposite corner of $c$. If $c^-$ is not a corner of a triangle, then the smoothing $\alpha'$ at $c$ cannot have small monogons or bigons and so it is taut by Proposition~\ref{prop:small}.
If $c^-$ is a corner of a triangle then $\alpha'$ is not taut, but $\alpha'$ cannot have monogons from a triangle or bigons from a quadrilateral because then $c$ would be a corner of a triangle or a quadrilateral. Thus, each small bigon of $\alpha'$ comes from a triangle with a corner at $c^-$, and the triangle and resulting bigon have the same combinatorial area.

Let $\Delta$ be a triangle of minimal combinatorial area among the triangles with corner $c^-$. Then, the resulting bigon has minimal combinatorial area, and by by Lemma \ref{lem: minimalarea} it is small. Therefore $\Delta$ is small (indeed, two vertices of $\Delta$ cannot be identified because $c$ would then be a corner of a triangle).
Now by Lemma~\ref{lem: moveable triangles}  
there is a triangle move of $\alpha$ that pushes the side of $\Delta$ opposite to $c^-$ across $\Delta$ to form a innermost triangle $\Delta'$ with corner $c$, and this move preserves tautness.
$\Delta'$ has an outer unoriented corner $c'$ with combinatorial angle larger than $c$.
We claim that $c'$ is not the corner of an $n$-gon with $n \leq 5$. If this is true then we are in a better position than when we started, and we can repeat the argument until we find an unoriented corner of $\alpha$ that is not the corner of a triangle or a quadrilateral, and whose opposite corner is not the corner of a triangle, so the smoothing at this corner is taut.

We only need the hypothesis that $c$ is not the corner of a pentagon to guarantee that a triangle $\widetilde \Delta$ above $\Delta$ does not enclose any lifting $\tilde c$ of $c$ (if it did, the two lines that cross at $\tilde c$ would split $\widetilde \Delta$ into polygons with at most 5 sides, and $\tilde c$ would be a corner of one of them).

Let $\alpha'$ be the result of the triangle move on $\alpha$.
$\alpha'$ is obtained by replacing an arc $a$ of $\alpha$ that contains a side of $\Delta$ by another arc $a'$ that runs parallel to the other two sides of $\Delta$, creating a triangle $\Delta'$ with corner $c$ that is not crossed by other arcs of $\alpha'$.
There is an isotopy of $\alpha$ that moves $a$ to $a'$
and takes place in a regular neighbourhood of $\Delta$ (we don't know if this isotopy can be done by a series of local triangle moves, but we don't need it).

Let $\widetilde \Delta$ be a triangle of $\widetilde \alpha$ above $\Delta$, let $\tilde c$ be a corner of $\widetilde \Delta$ above $c$ and $l_0$, $l_1$ and $l_2$ the lines of $\widetilde \alpha$ that contain its sides, where $l_0$ contains the side opposite to $\tilde c$.
If $\widetilde \Delta$ does not enclose any liftings of $\tilde c$, then the only line of $\widetilde \alpha$ that crosses $\tilde c$ during the isotopy is $l_0$, when the triangle $\widetilde \Delta'$ is created.

Now consider the outer unoriented corner $c'$ of $\Delta'$. We claim that this corner  still satisfies the hypotheses of the theorem.

\textbf{    $c'$ is not the corner of a triangle of $\alpha'$:}
    Otherwise $\tilde c'$ is the corner of a triangle formed by the lines $l_0'$ and $l_2'$ and another line $l_3'$ of $\widetilde \alpha$, but then $l_1'$, $l_2'$ and $l_3'$ also form triangle. Since $l_3$ did not cross $\tilde c$ during the isotopy, then $l_1$, $l_2$ and $l_3$ formed a triangle of $\widetilde \alpha$, contradicting the assumption that $c$ was not the corner of a triangle of $\alpha$.

\textbf{    $c'$ is not the corner of a quadrilateral of $\alpha'$:}
Otherwise $\tilde c'$ is the corner of a quadrilateral formed by the lines $l_0'$ and $l_2'$ and two other lines $l_3'$ and $l_4'$ of $\widetilde \alpha$. But then, as $l_3'$ and $l_4'$ do not cross $\widetilde \Delta'$, $l_1'$, $l_2'$, $l_3'$ and $l_4'$ also form a quadrilateral. (See Figure \ref{triangle3} again).
Since $l_3$ and $l_4$ did not cross $\tilde c$ during the isotopy, then $l_1$, $l_2$ $l_3$ and $l_4$ formed a quadrilateral of $\widetilde \alpha$, contradicting the assumption that $c$ was not the corner of a quadrilateral of $\alpha$.

 \textbf{$c'$ is not the corner of a pentagon of $\alpha'$:}
The proof is almost the same as for quadrilaterals, except that the assumption that 
$l_0'$, $l_2'$, $l_3'$, $l_4'$ and $l_5'$ form a pentagon of $\alpha'$ only implies that $l_1'$, $l_2'$, $l_3'$, $l_4'$ and $l_5'$ form a quadrilateral or a pentagon, so $l_1$, $l_2$, $l_3$, $l_4$ and $l_5$ form a quadrilateral or a pentagon of $\alpha$.

Thus, $c'$ satisfies the hypotheses of the  theorem, and is an unoriented corner with a larger combinatorial angle, contradicting our initial hypothesis.
\end{proof}

\begin{remark}
    The same proof as in Theorem~\ref{thm:mainsmooth} can also be used to obtain a slightly improved variation of Corollary~\ref{cor:maincurvessmooth}. Namely, that if $\Sigma$ is an orientable surface and $\alpha$ is a taut curve that admits a decomposition as $\alpha=\beta \gamma$ where $\beta$ and $\gamma$ intersect only at their endpoints, and no triangle of $\alpha$ that has $p=\gamma \cap \beta$ as a vertex is cyclically oriented, then the connected smoothing at $p$  is taut.   
\end{remark}

Using Theorem~\ref{thm:mainsmooth} and some elements of its proof, we can answer Question~\ref{qn:1 intro} for arcs with fixed endpoints completely; we also obtain an alternative proof of Theorem~\ref{prop:taut pair} for multi-arcs:

\begin{theorem}\label{thm: arcs}
Let $\Sigma$ be an orientable surface and let $\alpha$ be a non-simple multi-arc on $\Sigma$ that is taut fixing its endpoints.  Then  $\alpha$ has a smoothing that is taut fixing its endpoints and has no more components than $\alpha$. 
\end{theorem}

\begin{proof}
 Choose an orientation for $\alpha$.
 Let $a$ be an arc $a$ of $\alpha$ with intersection points, and let $a'$ be a subarc joining $\partial \Sigma \cap a$ to the first intersection point along $a$. Then one of the regions of $\Sigma - \alpha$ adjacent to $a'$, say $R$, must have an unoriented corner at $q$. Indeed, if the ``right side'' of $a'$ sees an oriented corner at $q$, then the ``left side'' of $a'$ must see an unoriented corner $c$ at $q$, as explained in Subsection~\ref{subsec: repel}.

Since $c$ is a corner of a region adjacent to the boundary of $\Sigma$, $c$ is not contained in any n-gon of $\alpha$, so Theorem~\ref{thm:mainsmooth}  shows that 
$\alpha$ has an unoriented smoothing without small monogons or bigons.

If $\alpha$ is a single arc, the unoriented smoothing is a single arc $\alpha'$, and if $\alpha'$ were not taut, then it would have a small monogon or bigon, so  $\alpha'$ must be taut.

Now assume that $\alpha$ has more than one arc. Since we started with an unoriented corner of a region $R$ adjacent to the boundary of $\Sigma$,
the proof of Theorem~\ref{thm:mainsmooth}  shows that we can find a sequence of triangle moves that transforms $\alpha$ into a taut arc $\alpha_1$ which has an unoriented corner $c_1$ that sees a region $R_1$ adjacent to the boundary of $\Sigma$, so
$c_1$ does not see any n-gon of $\alpha$ and its opposite corner $c_1'$ does not see a triangle. This shows that the unoriented smoothing at $c_1$ produces a multi-arc 
$\alpha_1'$ that does not have small monogons or bigons, but we still have to show that $\alpha_1'$ does not have large bigons.

A large bigon must traverse the same corner of the smoothing at least twice, but it cannot go over $c_1$ because $c_1$ is not contained in any polygonal disc bounded by $\alpha_1$. Hence, the bigon must go twice over $c_1'$, and so $c_1'$ must be the corner of a closed loop of $\alpha_1$. But $c_1'$ is an unoriented corner of an arc, so it is not the corner of a closed loop. 
\end{proof}

\begin{remark}
As mentioned in the introduction, there are taut arcs with arbitrarily large number of intersections, and where only one smoothing produces a taut arc.
The example in the centre of Figure~\ref{fig: onlyoneforarcs}, for instance, can be modified by making the arc travel alternately around the two shaded boundary components so that it locally looks like a figure $8$ curve. This increases the number of self-intersections, and produces examples with only one taut smoothing.
\end{remark}

It is not true that a non-simple arc that is taut with free endpoints has a smoothing that is taut with free endpoints. Indeed, a simple example  is given in Figure~\ref{fig: freearcs}.  However, the  best possible alternative holds:

\begin{figure}[h]
\centering
\includegraphics[height=1.7in]{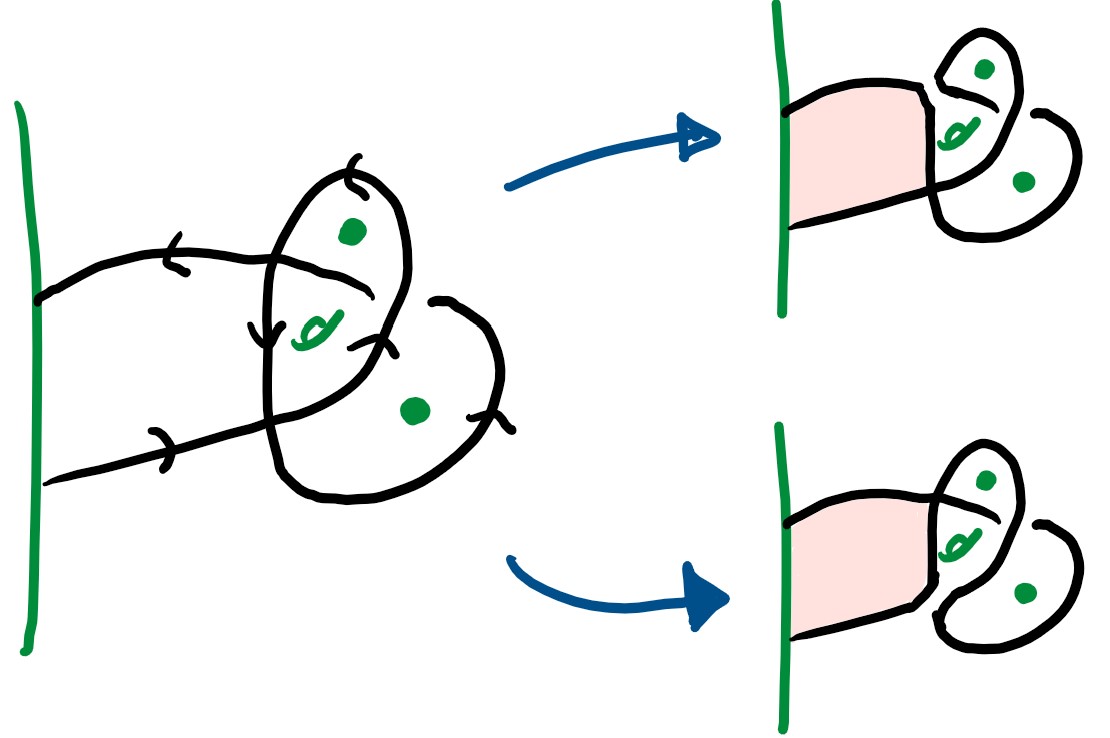}
\caption{A taut arc that cannot be smoothed to a taut arc if the endpoints are allowed to move freely on the boundary.}
\label{fig: freearcs}
\end{figure}

\begin{theorem}\label{cor:arcos libres}
    Let $\Sigma$ be an orientable surface and let $\alpha$ be a taut arc with free endpoints on $\Sigma$. If $\alpha$ has $k\geq 1$ self-intersections, then  $\alpha$ has a smoothing that cannot be homotoped to an arc with less than $k-2$ self-intersections. 
\end{theorem}

\begin{proof}
A 1-manifold in $\Sigma$ is taut with free endpoints if its preimage in $\widetilde \Sigma$ does not have any monogon, bigon or half-bigon.
We showed in Theorem~\ref{thm: arcs} that a taut arc $\alpha$ has a connected smoothing $\alpha'$ that does not have have any monogon or bigon.
We will show that $\alpha'$ has the desired property.
The proof is divided into two steps: first we show that if an arc is taut with fixed endpoints, then the half-bigons in the arc are small, then we show that a smoothing of an arc that is taut with free endpoints has at most one half-bigon.

\begin{claim}\label{clm:semismall}
  If an arc $\alpha'$ is taut with free endpoints, then the half-bigons of  $\alpha'$ are small.
\end{claim}

\begin{proof}
    
Assume that a pair of arcs $a$ and $a'$ above $\alpha$ in $\widetilde \Sigma$ form a half-bigon $S$.
Then $a$ and $a'$ have an endpoint at the same boundary component $b$ of $\widetilde \Sigma$, and $a'$ cannot be a translate of $a$ by a covering translation of $\widetilde \Sigma$ along $b$ (because those translates of $a$ must be disjoint).
Therefore if $a$ and $a'$ are oriented consistently with $\alpha$, then $S$ consists of an initial arc of $a$, and a final arc of $a'$ together with an arc of $b$.
If $S$ is not small, then the projections of $a$ and $a'$ overlap in $\alpha$, let
$o$ and $o'$ be two maximal subarcs of $a$ and $a'$ with the same image in $\alpha$.
Then $o$ is a final subarc of $a$ and $o'$ is an initial subarc of $a'$ that meet at the intersection point of $a$ an $a'$. Indeed,  the other possibilities would be that $o$ starts in $\partial \widetilde \Sigma$ or 
that $o'$ ends in $\partial \widetilde \Sigma$, but then $o'$ would start in $\partial \widetilde \Sigma$ or $o$ would end in  $\partial \widetilde \Sigma$, which is impossible. 

Let $g$ be the covering translation of $\widetilde \Sigma$ such that $a'=ga$.
As $\Sigma$ admits a hyperbolic metric that makes its boundary geodesic,
$\widetilde \Sigma$ is identified with a closed subset of the hyperbolic plane with geodesic boundary, and $g$ corresponds to a hyperbolic isometry that leaves some hyperbolic geodesic $l$ invariant.
Then $l$ is  the axis of $g$ and the endpoints of $l$ act as attracting and repelling points of the action of $g$ on the points at infinity of $\widetilde \Sigma$.

We will argue that if a half-bigon formed by $a$ and $a'$ is not small then the axis of $g$ cannot exist.

Since $a$ and $ga$ meet at one point, the axis of $g$ cannot meet $a$ or $ga$, and it must lie in one of the two complementary regions determined by $a$ and $ga$ with consistent orientations (see Figure~\ref{fig: semi1}).
For the same reason, as $ga$ and $g^2a$ meet at one point, the axis of $g$ cannot meet $ga$ or $g^2a$, and it must lie in one of the two complementary regions determined by $ga$ and $g^2a$ with consistent orientations.
But as $g^2a$ crosses $ga$ and does not end in $b$, it must also cross $a$, so the two regions where the axis of $g$ must lie are disjoint.
\end{proof}

\begin{figure}[h]
\centering
\includegraphics[height=1.7in]{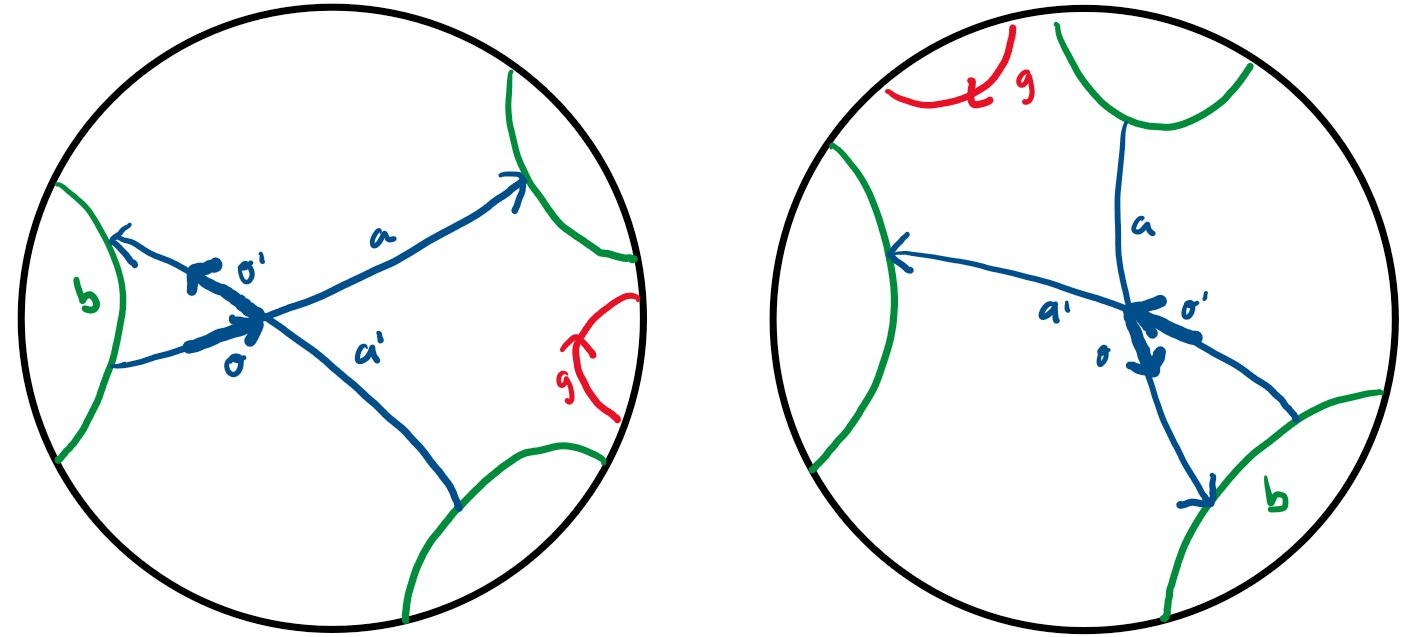}
\caption{Ingredients in the proof of Claim~\ref{clm:semismall}.}
\label{fig: semi1}
\end{figure}

\begin{claim}\label{clm:semionlyone}
    If an arc $\alpha$ is taut with free endpoints, and an arc $\alpha'$ is a smoothing of $\alpha$, then $\alpha'$ has at most one half-bigon.
\end{claim}

\begin{proof}
A half-bigon can only occur if $\alpha' $ (and therefore $\alpha$) has its endpoints in the same component $b$ of  $\partial \Sigma$, so we will assume that this is the case.

By Claim~\ref{clm:semismall}, any half-bigon $S$ of $\alpha'$ is small, so it goes over each corner of the smoothing at most once, and so $S$ is made of 3 or 4 subarcs of $\alpha$ and a subarc of $b$, and at least one of the corners of $S$ lies in an arc $a$ adjacent to $b$. Figure~\ref{fig: semi2} shows  the 3 possible ways in which  $S$ can arise.

\begin{figure}[h]
\centering
\includegraphics[height=1in]{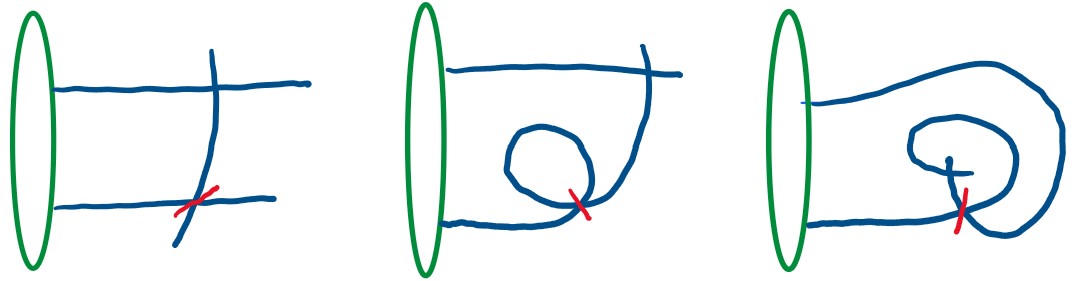}
\caption{Possible half-bigons after smoothing.}
\label{fig: semi2}
\end{figure}

Consider the universal cover $\widetilde \Sigma$.
Since $\alpha$ is taut with free endpoints then $\alpha$ has no half-bigons, so
the arcs above $\alpha$ adjacent to a line $\tilde b$ above $b$ are disjoint.
If $\alpha$ starts and ends in $b$, these arcs form two orbits under the action of the stabiliser $g$ of $\tilde b$, with the arcs
that start in $\tilde b$ alternating with the arcs that end in $\tilde b$.
If the smoothing $\alpha'$ has half-bigons, they lift to half-bigons adjacent to $\tilde b$.
We claim that the half-bigons are \emph{thin}, in the sense that each one lies between two consecutive arcs $a$ and $a'$ above $\alpha$. Indeed, if a half-bigon that has a corner of the smoothing at $a$ is not thin, then there is a half-bigon that starts at $a$ and ends at a translate $ga$, but this half-bigon cannot be small as its ends overlap in $\Sigma$.

Observe also that if there is a half-bigon on one side of $a$, then there can be no thin half-bigons on the other side of $a$. This follows from the fact that when viewed from $\tilde b$, the first corners of the smoothing along $a$ and $a'$  point to opposite sides.

Now, if there are two half-bigons between $a$ and $a'$, then they must overlap.  There are two possibilities:

\begin{enumerate}
\item  One half-bigon is contained in the other, as in the three left-most cases of Figure~\ref{fig: semi3}. If this happens, then the larger half-bigon must exit the region between the two adjacent arcs when it reaches the corner of the smaller half-bigon, which implies that it is not thin.

\item Each half-bigon has a corner of the smoothing that lies outside the other half-bigon, as in the three right-most cases of Figure~\ref{fig: semi3}. 
Then by considering all the possible configurations we can see that either
one of the corners of the smoothing is the inner corner of a triangle, 
or two corners of the smoothing are the corners of a trapezium. But as noted in Remark~\ref{rmk:connectedsmooth4casesreverse}, both scenarios are incompatible with the fact that $\alpha'$ is taut with fixed endpoints. 
\end{enumerate}
\end{proof}

\begin{figure}[h]
\centering
\includegraphics[height=1.05in]{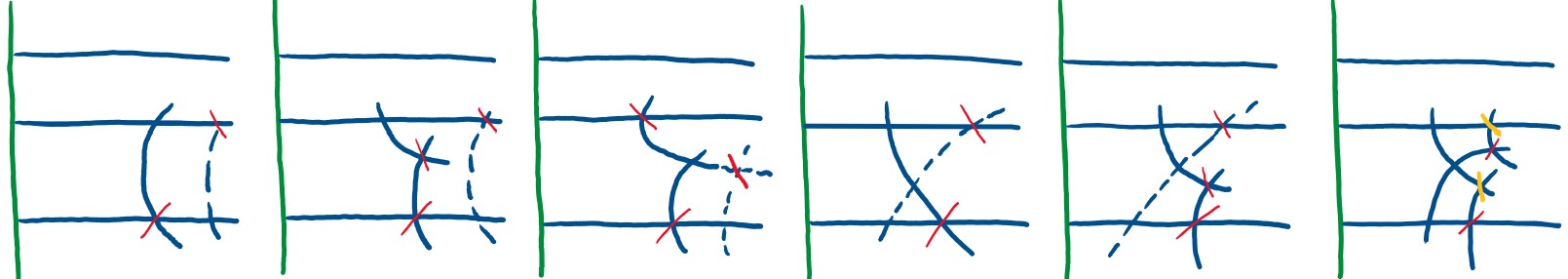}
\caption{Some potential  configurations in the proof of Claim~\ref{clm:semionlyone}.}
\label{fig: semi3}
\end{figure}

   It now follows that $\alpha'$ cannot be homotoped to have less than $k-2$ self-intersections, since performing a homotopy on $\alpha'$ that removes the small half-bigon removes exactly one self-intersection point and produces an arc with no half-bigons.
\end{proof}

\section{$k$-systoles and coarse orders on length spectra}\label{sec: geometry}

We now explain how to use the topological results obtained in Section~\ref{sec:taut smoothing} to extract geometric information.

If $[\alpha]$ is a free homotopy class of 1-manifolds on a compact surface $\Sigma$ and $g$ is a Riemannian metric on $\Sigma$ (that makes $\partial \Sigma$ convex if $\Sigma$ has boundary), we will denote by $\ell_g[\alpha]$ the minimum length among all the 1-manifolds in $[\alpha]$, and by $i[\alpha]$ the minimum number of self-intersections among all those 1-manifolds.

We begin with an application of Theorem~\ref{thm: arcs} and Theorem~\ref{cor:arcos libres}.

\begin{corollary}\label{cor: arcs geometryfull}
 For any orientable Riemannian surface $\Sigma$ with convex boundary and $\chi(\Sigma) < 0$, a shortest taut arc with fixed
  endpoints and at least $k$ self-intersections has exactly $k$ self-intersections. A  shortest taut arc with  free
  endpoints on $\partial \Sigma$ and at least $k$ self-intersections has at most $k+1$ self-intersections.
\end{corollary}

\begin{proof}
Assume that this is not the case, so that for some $k\geq 0$, a shortest taut arc $\alpha$ with endpoints $p$ and $q$ and at least $k$ self-intersections has $k+m$ self-intersections for some $m \geq 1$. By Theorem~\ref{thm: arcs}, $\alpha$ has a taut smoothing $\alpha'$, so $\alpha'$ has $k+m-1<k+m$ self-intersections. As $\alpha'$ is obtained from cutting and rejoining two arcs of $\alpha$ (which does not change the length) and rounding-out the resulting corners (which reduces length), then $l_g[\alpha']<l_g[\alpha]$, contradicting our initial hypothesis. 
The proof of the second assertion follows by  the same reasoning as before, but applying Theorem~\ref{cor:arcos libres} rather than Theorem~\ref{thm: arcs}.
\end{proof}

The conclusion of Corollary~\ref{cor: arcs geometryfull} is optimal:

\begin{proposition}\label{prop:noshortarc}
    The sphere with 3 holes admits a Riemannian metric for which the shortest arc with free endpoints and at least 3 self-intersections has 4 self-intersections. 
\end{proposition}

\begin{proof}

Let $\alpha$ be the arc in the left of Figure~\ref{fig:noshortarc}.
    Since $\alpha$ is taut with free endpoints, by~\cite[Theorem 1.1]{NC01} there is a Riemannian metric on $\Sigma$ for which $\alpha$ is the shortest geodesic in its free homotopy class (with free endpoints), and every geodesic arc in $\Sigma$ shorter than $\alpha$ lies in a regular neighbourhood $R$ of $\alpha$.

As $\alpha$ has $4$ self-intersections, $\alpha$ is made of $9$ segments. Every arc in $R$ is homotopic to one made of copies of the segments of $\alpha$. The metric can be chosen so that each segment has length exactly $1$ and every geodesic arc in $R$ made of $n$ segments  has length larger than $n(1-\epsilon)$ for any prescribed $\epsilon>0$.

To see that $\alpha$ is the shortest geodesic arc with at least $3$ self-intersections, it is enough to prove that there are no  arcs in $R$ that are taut with free endpoints,  have $3$ self-intersections, and are made of less than $10$ segments of $\alpha$. Since $\alpha$ has no smoothings that are taut with free endpoints, such an arc must go over some segments of $\alpha$ more than once. Analysing all the possibilities, the shortest arc with 3 self-intersections that is taut with free endpoints is the arc $\beta$ shown in the right of  Figure~\ref{fig:noshortarc}. Now, $\beta$ uses $10$ segments of $\alpha$, and therefore is longer than $\alpha$.
\end{proof}

\begin{figure}[h]
\centering
\includegraphics[height=.9in]{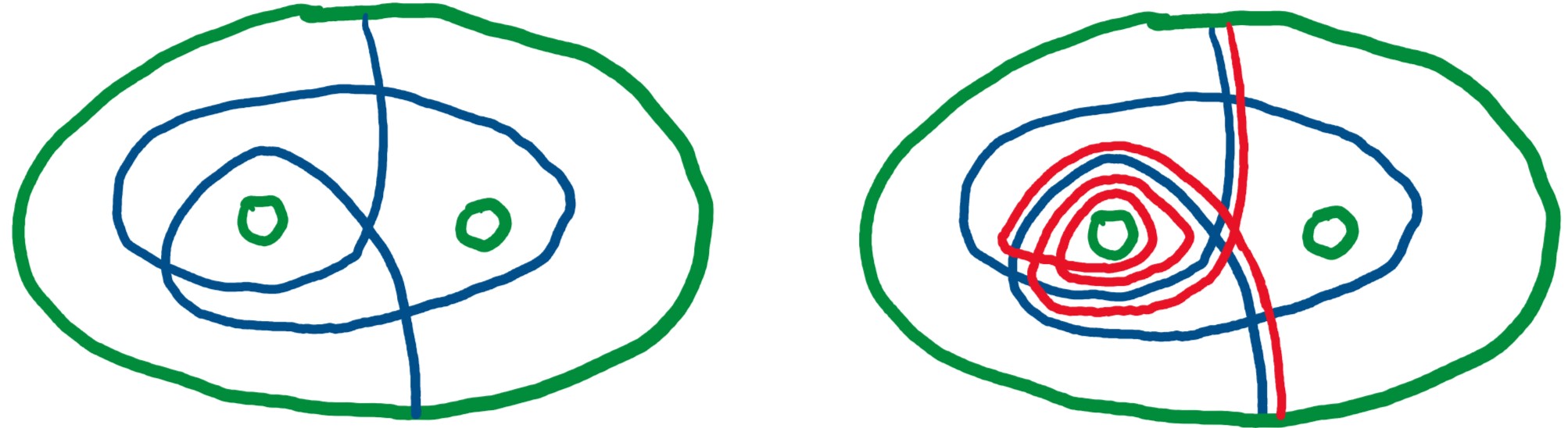}
\caption{The arc $\alpha$ (blue) and the arc $\beta$ (red) in the proof of Proposition~\ref{prop:noshortarc}.}
\label{fig:noshortarc}
\end{figure}

We note that the argument in Proposition~\ref{prop:noshortarc} uses a Riemannian metric that is not hyperbolic. Thus, the following remains unresolved:

\begin{question}
Let $g$ be a hyperbolic metric on the 3-holed sphere $\Sigma$.
     Does a  shortest taut arc with  free endpoints on $\partial \Sigma$ and at least $3$ self-intersections have exactly $3$ self-intersections?
\end{question}

 The proof of the first assertion in Corollary~\ref{cor: arcs geometryfull} actually shows:

\begin{corollary}\label{cor:universalarc}
In an orientable surface $\Sigma$, for every free homotopy class $[\alpha]$ of non-simple arcs with fixed endpoints, there exists another free homotopy class $[\beta]$ with the same endpoints and with $i[\beta]=i[\alpha]-1$
such that 
$$\ell_g[\beta] < \ell_g[\alpha]$$ for every Riemannian metric $g$ on $\Sigma$.
\end{corollary}

\begin{proof}
    A taut smoothing of a taut representative of $\alpha$ yields the desired free homotopy class.
\end{proof}

The analogous result for homotopy classes of arcs with free endpoints fails, as can be seen by choosing different metrics in Figure~\ref{fig: freearcs}.
In this case, we instead get:

\begin{corollary}\label{cor:universalarcfree}
In an orientable surface $\Sigma$, for every homotopy class $[\alpha]$ of arcs with free endpoints, there exists another homotopy class $[\beta]$ with $i[\beta]=i[\alpha]-1$ (if the endpoints of $[\alpha]$ lie on different boundary components of $\Sigma$) or $i[\beta]=i[\alpha]-2$ (if the endpoints of $[\alpha]$ lie on the same boundary component), such that 
$$\ell_g[\beta] < \ell_g[\alpha]$$ for every Riemannian metric $g$ on $\Sigma$.

\end{corollary}

\begin{proof}
The second part of the claim follows from Theorem~\ref{cor:arcos libres}, for the same reason as the proof above.
    For the first part of the claim, note that the proof of Theorem~\ref{cor:arcos libres} actually shows that a taut arc with endpoints on distinct components of $\partial \Sigma$ has a smoothing that is taut without fixing the endpoints. Indeed, a smoothing of such an arc at an unoriented corner forming a maximal combinatorial angle cannot produce a half-bigon.
\end{proof}

By iteratively applying Corollary~\ref{cor:arcos libres}, we deduce the existence of certain sequences of homotopy classes that are simultaneously ordered by intersection number and length in every Riemannian metric. We record this as a remark:

\begin{remark}\label{cor:orthogeointro}
     Let $\Sigma$ be an orientable surface. 
 For every free homotopy class of arcs $[\alpha]$ there is a  sequence $\{[\alpha]:=[\alpha_1], \ldots, [\alpha_n]\}$   where  $[\alpha_n]$ is simple, and such that
 \begin{enumerate}
     \item for free homotopy classes with fixed endpoints, $i[\alpha_{j-1}]-1 = i[\alpha_j]$ and $\ell_g[\alpha_j]< \ell_g[\alpha_{j-1}]$ for each $j \in \{1, \ldots, i[\alpha]\}$, and
     \item for free homotopy classes free with free endpoints, then  $i[\alpha_{j-1}]-1 \geq i[\alpha_j]\geq i[\alpha_{j-1}]-2$ and $\ell_g[\alpha_j]< \ell_g[\alpha_{j-1}]$ for each $j \in \{1, \ldots, i[\alpha]\}$
 \end{enumerate} for every Riemannian metric $g$.
\end{remark}

\begin{remark}
    For a Riemannian metric on $\Sigma$ that makes the boundary convex, each homotopy class of arcs with free endpoints in  $\partial \Sigma$ contains a shortest representative that is perpendicular to the boundary. 
    Corollary~\ref{cor:universalarcfree}  
    applies to the lengths of these orthogeodesics in $\Sigma$. 
\end{remark}



Let $\mathcal{C}'$ be the free homotopy classes of curves in $\Sigma$ having a minimal representative satisfying the  condition in Theorem~\ref{cor:maincurvessmooth}.
From Theorem~\ref{thm:mainsmooth}, we  deduce a partial analogue to Corollary~\ref{cor:universalarcfree}:

\begin{corollary}\label{cor:maincurvessmooth2}
     Let $\Sigma$ be an orientable surface. Then for every free homotopy class of curves $[\alpha]$ in $\mathcal{C'}$ there exists a class $[\beta]$ with $i[\beta]=i[\alpha]-1$ and $$\ell_g[\beta] < \ell_g[\alpha]$$
    for every Riemannian metric on $\Sigma$.
\end{corollary}

\begin{proof}
Theorem~\ref{thm:mainsmooth} implies that $\alpha$ has a taut smoothing $\alpha'$. Hence $i[\alpha']=i[\alpha]-1$ and $\ell_g[\alpha'] < \ell_g[\alpha]$.
\end{proof}

\section{Further directions}\label{sec: further}

Unlike in Theorems~\ref{thm: arcs intro} and~\ref{thm:taut smooth curve if no trisq intro}, the method of proof used in Theorem~\ref{thm:multi-intro} is not constructive. It would be interesting to see whether one can extract an algorithm to identify taut smoothings for the 1-manifolds satisfying the conditions of Theorem~\ref{prop:taut pair}.

\begin{prob}\label{qn: algorithmintro}
Let $\alpha$ be a primitive  1-manifold that is taut with fixed endpoints and has at least two intersecting components. Construct an algorithm that finds the taut smoothings of $\alpha$. 
\end{prob}

 The approach in this paper requires one to be able to move between distinct minimal configurations of a free homotopy class. This adds flexibility, but creates some technical difficulties. Lemma~\ref{lem: smoothing invariants htpy class} shows that  if there is a taut smoothing in some configuration, then there is a taut smoothing in every minimal configuration, and in principle this smoothing can be found by keeping track of the triangle moves. From an algorithmic point of view, it would be preferable to be able to detect taut smoothings without having to change the configuration:

\begin{question}
     Is there an effective way to identify taut smoothings without performing any homotopies? In particular, is there a way to identify an unoriented corner that is not the corner of triangles, rhombi, or trapeziums?
 \end{question}

Part of the difficulty with finding taut smoothings of curves (arcs, etc...) lies in the fact that an arbitrarily chosen smoothing can be quite far from being taut. In particular, as mentioned in the introduction, there are examples of curves and arcs which have only one taut smoothing, or which have smoothings that remove all self-intersections (see Figure~\ref{fig: onlyoneforarcs}). 
Perhaps a more tractable approach to understanding the behaviour of smoothings is probabilistic:

\begin{question}\label{qn: random}

What proportion of the self-intersections does a ``random smoothing" of a taut $1$-manifold (for some sensible notion of randomness) remove?
\end{question}

We conjecture that, generically, taut smoothings exist:

\begin{conjecture}
 A random curve (for instance, in the sense of~\cite{AougabGaster22}) has a taut smoothing.
\end{conjecture}

For multi-arcs with free endpoints,  it seems  hard to carry over the proof of Theorem~\ref{cor:arcos libres} without having to rule out various problematic scenarios that do not arise in the  case of a single arc. Nevertheless, we expect that the following holds: 

\begin{conjecture}
If $\alpha$ is a  multi-arc  with $n$ arc components and  $k\geq 1$ self-intersections, and  $\alpha$ is taut with free endpoints,  then $\alpha$ has a smoothing that cannot be homotoped to a multi-arc with less than $k-(n+1)$ self-intersections. 
\end{conjecture}

Finally, for our methods to work, it is essential to assume that $\Sigma$ be oriented. It seems likely that Question~\ref{qn:1 intro} has a negative answer without this hypothesis:

\begin{question}\label{qn: orientable?}
If $\Sigma$ is non-orientable, 
do there exist taut curves or arcs on $\Sigma$ that do not admit any taut smoothings? 
\end{question}

\bibliographystyle{alpha}
\bibliography{bib9.bib}

\end{document}